\documentclass[10pt,a4paper]{amsart}

\usepackage[utf8]{inputenc}
\usepackage{t1enc}

\usepackage{amsfonts}
\usepackage{nccmath}
\usepackage{mathtools}

\usepackage{xcolor}

\usepackage[margin=2.7cm,marginpar=2cm]{geometry}

\usepackage{CJKutf8}

\usepackage[unicode]{hyperref}
\numberwithin{equation}{section}

\def\equationautorefname~#1\null{Equation~(#1)\null}

\makeatletter

\newmuskip\pFqmuskip
\newcommand*\pFq[6][8]{%
  \begingroup 
  \pFqmuskip=#1mu\relax
  \mathchardef\normalcomma=\mathcode`,
  \mathcode`\,=\string"8000
  \begingroup\lccode`\~=`\,
  \lowercase{\endgroup\let~}\pFqcomma
  {}_{#2}F_{#3}{\left[\genfrac..{0pt}{}{#4}{#5};#6\right]}%
  \endgroup
}
\newcommand{\pFqcomma}{{\normalcomma}\mskip\pFqmuskip}

\makeatother

\usepackage{thmtools}

\makeatletter
\def\@endtheorem{\endtrivlist}
\makeatother

\declaretheorem[
style=plain,
name=Theorem,
numbered=yes,
numberwithin=section,
refname={Theorem,Theorems},
Refname={Theorem,Theorems}
]{theorem}

\declaretheorem[
style=definition,
name=Remark,
numberlike=theorem,
refname={Remark,Remarks},
Refname={Remark,Remarks}
]{remark}

\declaretheorem[
style=definition,
name=Example,
numberlike=theorem,
refname={Example,Examples},
Refname={Example,Examples}
]{example}
\declaretheorem[
style=plain,
name=Corollary,
numberlike=theorem,
refname={Corollary,Corollaries},
Refname={Corollary,Corollaries}
]{corollary}
\declaretheorem[
style=plain,
name=Lemma,
numberlike=theorem,
refname={Lemma,Lemmas},
Refname={Lemma,Lemmas }
]{lemma}

\usepackage{shuffle} 

\DeclareFontFamily{U}{shuffle}{}
\DeclareFontShape{U}{shuffle}{m}{n}{%
	  <1-8>shuffle7%
	<8->shuffle10%
}{}

\usepackage{xfrac}
\newcommand{\half}{{\sfrac{1\!}{2}}}

\newcommand{\abs}[1]{\left| #1 \right|}
\usepackage[normalem]{ulem}
\renewcommand{\vec}[1]{\uline{\boldsymbol{\mathbf{#1}}}}
\DeclareMathOperator{\Li}{Li}
\DeclareMathOperator{\reg}{reg}

\newcommand{\ii}{\mkern 1mu \mathrm{i} \mkern 1mu}
\newcommand{\dd}{\mkern 1mu \mathrm{d} \mkern 1mu}

\usepackage{listings}
\definecolor{codegreen}{rgb}{0,0.6,0}
\definecolor{codegray}{rgb}{0.5,0.5,0.5}
\definecolor{codepurple}{rgb}{0.58,0,0.82}
\definecolor{backcolour}{rgb}{0.95,0.95,0.92}
\lstdefinestyle{mystyle}{
	aboveskip=\bigskipamount,
	belowskip=\bigskipamount,
	commentstyle=\color{codegreen},
	keywordstyle=\color{magenta},
	numberstyle=\tiny\color{codegray},
	stringstyle=\color{codepurple},
	basicstyle=\ttfamily\small,
	breakatwhitespace=false,         
	breaklines=true,                 
	captionpos=b,                    
	keepspaces=true,                 
	numbers=left,                    
	numbersep=5pt,                  
	showspaces=false,                
	showstringspaces=false,
	showtabs=false,                  
	tabsize=2
}
\def\+{\!+\!}
\def\-{\!-\!}

\newcommand\poch[2]{\left(  #1 \right)_{#2}}

\let\overlineO\overline
\renewcommand{\overline}[1]{\overlineO{\mathclap{\phantom{I}}#1}}

\usepackage{accsupp}
\DeclareRobustCommand\squelch[1]{%
	\BeginAccSupp{method=plain,ActualText={}}#1\EndAccSupp{}}

\newcounter{inprcnt}

\newcommand{\outpr}{\textcolor{gray!80!white}{\squelch{\smaller \texttt{Out[\arabic{inprcnt}]= }}}}

\usepackage{environ}

\makeatletter%
\NewEnviron{mathout}{%
	\small%
	\setlength\@mathmargin{0pt}%
	\setlength{\abovedisplayskip}{\bigskipamount}%
	\setlength{\belowdisplayskip}{\bigskipamount}%
	\setlength{\abovedisplayshortskip}{\bigskipamount}%
	\setlength{\belowdisplayshortskip}{\bigskipamount}%
	\begin{flalign}%
	\BODY
   \end{flalign}%
}
\makeatother%

\usepackage{stackengine}
\newcommand\brabar{\scalebox{.3}{(\,}\raisebox{-2.1pt}{--}\scalebox{.3}{\,)}}

\makeatletter
\@namedef{subjclassname@2020}{%
	\textup{2020} Mathematics Subject Classification}
\makeatother

\begin{document}
	
	\title[Creative telescoping and generating functions of MZV's]{Creative telescoping and generating functions \\ of (variants of) multiple zeta values}
	\date{25 April 2024}
	
	\author[Kam Cheong Au]{Kam Cheong Au}
	
	\address{Rheinische Friedrich-Wilhelms-Universität Bonn \\ Mathematical Institute \\ 53115 Bonn, Germany} 
 	\email{s6kmauuu@uni-bonn.de}

       \author[Steven Charlton]{Steven Charlton}
		\address{Max Planck Institute for Mathematics, Vivatsgasse 7, Bonn 53111, Germany}
	\email{charlton@mpim-bonn.mpg.de}

	

	\begin{abstract} 
		We show how to convert the generating series of interpolated multiple zeta values, or multiple $t$ values, with repeating blocks of length 1 into hypergeometric series. Then we invoke creative telescoping on their generating functions, in some known cases for illustration, and in some apparently new cases, reducing them to polynomials in Riemann zeta values. The new evaluations, including \( \zeta^\half(\{\overline{2}\}^n,3) \), \( \zeta^\star(\{1,3\}^n,1,2) \) and \( t^\half(2,\{1\}^n,2) \), resolve some questions raised elsewhere, and seem to be non-trivial using other methods.
	\end{abstract}

	\maketitle

\section{Introduction}

	Multiple zeta values (MZV's) are a multivariable analogue of the Riemann zeta values, defined by the following nested sum
	\begin{equation}\label{eqn:mzvdef}
		\zeta(s_1,\ldots,s_d) \coloneqq \sum_{1 \leq n_1 < n_2 < \cdots < n_d} \frac{1}{n_1^{s_1} \cdots n_d^{s_d}} \,,
	\end{equation}
	where \( s_1,\ldots,s_d \) are positive integers, and \( s_d \geq 2 \) for convergence.  These values appear in many areas of mathematics, and mathematical physics (associators and knot invariants \cite{LeMurakamiMZVAssoc}, Feynman amplitudes \cite{BroadhurstKConj96}, periods of mixed Tate motives \cite{BrownMTM12}).  
	
	Multiple zeta values have a rich algebraic structure, with many known types of relations (including extended double shuffle relations \cite{IKZ06}, associator relations \cite{RacinetDoubleShuff}, period polynomial relations \cite{GKZ06}, or confluence relations \cite{HIROSEConfluence}) and special evaluations (such as of \( \zeta(1,3,\ldots,1,3) \) \cite{zagierValues94},  \cite[\S11]{BBBLspec} which foreshadowed the block decomposition \cite{CharltonThesis16,CharltonBlock21}, block filtration \cite{KeilthyThesis20,KeilthyBlock21}, and block shuffle product \cite{HSBlockShuffle}; of \( \zeta(2, \ldots, 2, 3, 2, \ldots, 2) \) \cite{zagier2232} used by Brown to establish the Hoffman basis \cite{BrownDecomposition12,BrownMTM12} of motivic MZV's and to describe the periods of the category of mixed Tate motives over \( \mathbb{Z} \); or of \( \zeta(2, \ldots, 2, 4, 2, \ldots, 2) \) \cite{CK2242} which was used to show the period polynomial relations \cite{GKZ06} follow from the block relations describing the reduced block Lie algebra \cite{KeilthyThesis20,KeilthyBlock21}.  One is readily lead to investigate other such special evaluations of MZV's (and many of natural generalisations and variants) as an avenue to uncover deeper structures on the algebra of MZV's and related objects. \medskip
	
	Beyond the `classical' type of MZV introduced above, the one of the most prominent variants is the so-called alternating MZV's \cite{BBBkfold,BroadhurstKConj96} (also called Euler-Zagier sums, which allow an alternating sign \( (-1)^{n_i} \) to appear in the numerator of \autoref{eqn:mzvdef}) or more generally the cyclotomic (or coloured) MZV's of level \( N \) (which allow a character \( \sigma_i^{n_i} \), \( \sigma_i^N = 1 \), with alternating MZV's being the case \( N = 2 \)).  In another direction, by relaxing the condition \( n_i < n_{i\+1} \) to \( n_i \leq n_{i\+1} \) and introducing a weighting when \( n_i = n_{i+1} \) we can define so-called interpolated MZV's \cite{YamamotoInterp}, which bridge between the MZV's and the multiple zeta star values (MZSV's) \( \zeta^\star \), which are obtained by replacing \( 1 \leq n_1 < n_2 < \cdots < n_d \) by \( 1 \leq n_1 \leq n_2 \leq \cdots \leq n_d \) in \autoref{eqn:mzvdef}.
	
	Naturally one can combine these two variants, to define the interpolated cyclotomic MZV's as follows.  For \( s_1,\ldots,s_d \) integers, and \( \sigma_1,\ldots,\sigma_d \in \{ z \in \mathbb{C} \colon \abs{z} = 1 \} \), define
	\begin{equation}\label{eqn:zetar}
		\zeta^r(s_1,\ldots,s_d ; \sigma_1,\ldots,\sigma_d) \coloneqq \sum_{1 \leq n_1 \leq n_2 \leq \cdots \leq n_d} r^{\delta_{n_1,n_2} + \delta_{n_2,n_3} + \cdots + \delta_{n_{d\-1},n_d}} \, \frac{\sigma_1^{n_1} \cdots \sigma_d^{n_d}}{n_1^{s_1} \cdots n_d^{s_d}} \,,
	\end{equation}
	where \( \delta_{a,b} \) is the Kronecker delta, namely \( \delta_{a,b} = 1 \) if \( a = b \), and 0 otherwise.  By considering whether each non-strict inequality \( n_i \leq n_{i\+1} \) is \( n_i < n_{i\+1} \) or \( n_i = n_{i\+1} \), we can expand \( \zeta^r \) as a sum of \( 2^{d-1} \) cyclotomic MZV's.  For example
	\begin{align*}
		\zeta^r(s_1,s_2,s_3; \sigma_1,\sigma_2,\sigma_3) \, = \,  {} 
		& \zeta(s_1,s_2,s_3; \sigma_1,\sigma_2,\sigma_3) + 
		r \cdot  \zeta(s_1+s_2,s_3; \sigma_1\cdot\sigma_2,\sigma_3)  \\
		& + r \cdot \zeta(s_1,s_2+s_3; \sigma_1,\sigma_2\cdot\sigma_3) + 
		r^2 \cdot  \zeta(s_1+s_2+s_3; \sigma_1\cdot\sigma_2\cdot\sigma_3) \,;
	\end{align*}
	note that indices are summed, while characters are multiplied when considering \( n_i = n_{i\+1} \) and contracting those terms.  From this, we also see \( \zeta^r(s_1,\ldots,s_d; \sigma_1,\ldots,\sigma_d) \) converges for \( (s_d,\sigma_d) \neq (1,1) \) (c.f. \cite[Corollary 2.3.10]{ZhaoBook}, when \( r = 0 \), and expand \( \zeta^r \) in terms of \( \zeta \), for the general case).  In the case of (interpolated) alternating MZV's, it is common to simplify the notation by decorating \( s_i \) with a bar (i.e. \( \overline{s_i} \)) if and only if \( \sigma_i = -1 \), for example
	\[
		\zeta^r(a,b,c,d,e; 1,-1,-1,1,-1) = \zeta^r(a,\overline{b},\overline{c},d,\overline{e}) \,.
	\] 
	
	Certain special cases of interpolated MZV's stand out for being of particular interest, including the case \( r = 0 \) which recovers the `classical' MZV's \( \zeta^0 = \zeta \), and the case \( r = 1 \), which gives the so-called multiple zeta star values (MZSV's) \( \zeta^1 = \zeta^\star \).  At the midpoint \( r = \half \), the multiple zeta half values (MZ$\half$V's) have a particularly simple stuffle-product structure, starting with \( \zeta^\half(a)\zeta^\half(b)= \zeta^\half(a,b) + \zeta^\half(b,a) \). which involves no depth 1 terms (c.f. \cite[\S3]{HoffmanQuasi2020}).  More significantly, Zhao's generalised two-one theorem \cite{ZhaoIdentity16} gives a beautiful relation expressing any non-alternating MZSV's in terms of certain special alternating MZ$\half$V's (those with \( \sigma_i = -1 \) if and only if \( s_i \) is even, i.e. exactly the even indices are barred), and vice versa.  (A rephrasing of this bijection via the block decomposition was given by the second named author in \cite{CharltonCyclic20}.)  For example, Zhao's generalised two-one theorem gives
	\[
		\zeta^\star(1,1,2,2,3,3,2) = 2^{4}\zeta^\half(1,\overline{6},3,\overline{4}) \,.
	\]
	In one of the examples below (\autoref{sec:zs111222}), we shall make use of this theorem to rewrite, for \( a,b \geq 1 \):
	\[
		 \zeta^\star(\{1\}^a, \{2\}^b) = 2^a \zeta^\half(\{1\}^{a-1},2b+1) \,,
	\]
	in order to be able to write the generating series of this MZSV via hypergeometric functions, and apply the ideas of creative telescoping.
	Here we have employed the notation 
	\[
	\{x_1,\ldots,x_i\}^n = {\overbrace{x_1,\ldots,x_i, \, \ldots \,, x_1,\ldots,x_i}^{\text{$n$ repetitions}}} \,,
	\]
	which means exactly \( n \) repetitions of the string \( x_1,\ldots,x_i \).  We will likewise (\autoref{sec:z2232}) use the generalised two-one theorem to write
	\[
		\zeta^\star(\{1,3\}^n, 1,2) = 2^{2n+1} \zeta^\half(\{\overline{2}\}^{2n}, 3) \,,
	\]
	in order to apply creative telescoping to the resulting hypergeometric generating series, and obtain evaluations for this multiple zeta star value.
	
	In this note, we henceforth focus on the cases \( r = 0, \half, 1 \) of interpolated MZV's. \medskip
	
	A final object of interest is the (interpolated) multiple \( t \) values (MtV's) \cite{HoffmanOdd19}, which changes \autoref{eqn:zetar} to involve a sum only over odd denominators.  Namely
	\begin{equation}\label{eqn:tr}
		t^r(s_1,\ldots,s_d) \coloneqq \sum_{1 \leq n_1 \leq n_2 \leq \cdots \leq n_d} \frac{r^{\delta_{n_1,n_2} + \delta_{n_2,n_3} + \cdots + \delta_{n_{d\-1}, n_d}}}{(2n_1 - 1)^{s_1} \cdots (2n_d - 1)^{s_d}} \,.
	\end{equation}
	(Naturally, one can define alternating MtV's; for the current note, we do not need to consider them.)\medskip

	The main goal of this note is to explain and apply the ideas of creative telescoping in the context of multiple zeta values.  It seems that these ideas, and in particular the computer algebra implementation available in Christoph Koutschan's \textsf{HolonomicFunctions} package \cite{Koutschan09} for \textsf{Mathematica}, are not as well-known as they should be in the MZV context; many conjectural and/or non-trivial multiple zeta value identities can be readily investigated and proven using these techniques, which otherwise are non-trivial to handle.
	
	This note is structured as follows.  We start with some preliminaries: In \autoref{sec:genser} we will explain how to write generating series of interpolated multiple zeta values with length 1 repeating blocks via hypergeometric series, which are then amenable to creative telescoping.  In \autoref{sec:ct} we will review the ideas of creative telescoping, and how to use the \textsf{HolonomicFunctions} package.  Then in \autoref{sec:analysis} we will cover some analytic results necessary to fix the coefficients functions which appear after solving recurrence relations obtained via creative telescoping.
	
	We then start by applying creative telescoping to some known examples, for illustration and background.  We treat Chen and Eie's \cite{ChenEie} evaluation of \( \zeta^\star(1,\ldots,1,2,\ldots,2) \) in \autoref{sec:zs111222},  Zagier's \cite{zagier2232} evaluation of \( \zeta(2, \ldots, 2, 3, 2, \ldots, 2) \) in \autoref{sec:z2232}, and Murakami's \cite{murakami21} evaluation of \( t(2,\ldots,2,3,2,\ldots,2) \) in \autoref{sec:t2232}.  Next, in \autoref{sec:zh223} we turn to evaluations of \( \zeta^\half(\{\overline{2}\}^n, 3) \) and \( \zeta^\half(3,\{\overline{2}\}^n) \) which follow from a seemingly non-trivial \( {}_6F_5 \) hypergeometric identity (\autoref{theorem:6F5identity}).  Via Zhao's generalised two-one theorem we then obtain some apparently new evaluations for multiple zeta star values involving \( \{1,3\}^m \).  Finally in \autoref{sec:th21112} we establish an evaluation conjectured by Hoffman and the second named author \cite[Conjecture 3.10]{charltonHoffmanMtVSym}, on \( t^\half(2,\{1\}^n,2) \). \medskip
	
	\paragraph{\bf Acknowledgements}  SC is grateful to the Max Planck Institute for Mathematics, Bonn, for support, hospitality and excellent working conditions.

\section{Preliminaries}

We start with some necessary preliminaries: how to write the generating series of MZV's in a way which we can apply the creative telescoping framework, how the creative telescoping framework works algorithmically, and some analytic properties we need to fix unknown coefficients in the evaluations we derive.

\subsection{Generating functions of MZV variants}\label{sec:genser}

We gather some examples of generating series of (interpolated) MZV's, which we shall employ later in the creative telescoping framework. \medskip

We start with generating series which involve a single repeating block of length 1; these follow easily from the definitions of the MZV's and MZSV's.  We have
\begin{align*}
	\sum_{j\geq 0} \zeta(\{a\}^j) x^j &= \prod_{n\geq 1} \Big(1+\frac{x}{n^a}\Big) \,, \\
	\sum_{j\geq 0} \zeta^\star(\{a\}^j) x^j &= \prod_{n\geq 1} \Big(1+\frac{x}{n^a} + \frac{x^2}{n^{2a}} + \cdots\Big) 
		= \prod_{n\geq1} \Big( 1 - \frac{x}{n^a} \Big)^{-1} \,.
\end{align*}
The cases \( \zeta \), and \( \zeta^\star \) are somehow better known (c.f. Equation (11), (44) \cite{zagier2232}), but one can readily generalise this to the interpolated MZV's, to obtain
\[
	\sum_{j\geq 0} \zeta^r(\{a\}^j) x^j = \prod_{n\geq 1} \Big(1+\frac{x}{n^a} + \frac{r x^2}{n^{2a}} + \frac{r^2 x^3}{n^{3a}} + \cdots \Big) 
	= \prod_{n\geq1} \biggl(  
		\frac{1 + (1-r) \cdot \frac{x}{n^a}}{1 - r \cdot \frac{ x}{n^a}}
	\biggr) \,.
\]
Restricting to \( n = 2m-1 \) odd immediately gives a generating series for the interpolated MtV's,
\begin{align*}
	\sum_{j\geq 0} t^r(\{a\}^j) x^j & = \prod_{m\geq 1} \Big(1+\frac{x}{(2m-1)^a} + \frac{r x^2}{(2m-1)^{2a}} + \frac{r^2 x^3}{(2m-1)^{3a}} + \cdots \Big)  \\
	& = \prod_{m\geq 1} \biggl(  
	\frac{1 + (1-r) \cdot \frac{x}{(2m-1)^a}}{1 - r \cdot \frac{ x}{(2m-1)^a}}
	\biggr) \,.
\end{align*}

Similarly, by introducing the character \( (-1)^n \), we have the generating series for the alternating MZV's
\begin{align*}
	\sum_{j\geq 0} \zeta^r(\{\overline{a}\}^j) x^j & = \prod_{n\geq 1} \biggl( 1+\frac{(-1)^n x}{n^{a}} +  \frac{ r (-1)^{2n} x^2}{n^{2a}}  +  \frac{r^2 (-1)^{3n} x^3}{n^{3a}} + \cdots \biggr) \\
	& = \prod_{n\geq1} \biggl(
		\frac{1 + (1-r) \cdot \frac{(-1)^n x}{n^a}}{1 - r \cdot \frac{(-1)^n x}{n^a}}
		\biggr) \,.
\end{align*}
Likewise, one could obtain a generating series for arbitrary interpolated cyclotomic MZV's.
\medskip

More generally, we can write the generating function of any family of (interpolated, alternating) MZV's or (interpolated) MtV's, which have repeating blocks of length 1, via some type of hypergeometric series.  For example
\begin{align*}
	 \sum_{j,k\geq0} \zeta^r(\{2\}^j, 3, \{2\}^k) \, x^j y^k  
	& = \sum_{\ell \geq 1} \prod_{n<\ell}  \begin{aligned}[t]
	& \Bigl( 1 + \frac{x}{n^2} + \frac{r x^2}{n^4} + \cdots \Bigr) \cdot \Big( 1 + \frac{r x}{\ell^2} + \frac{r^2 x^2}{\ell^4} + \cdots \Big) \\ 
	& \quad \cdot \frac{1}{\ell^3}  \cdot \Big( 1 + \frac{r y}{\ell^2} + \frac{r^2 y^2}{\ell^4} + \cdots \Big)  \cdot \prod_{n>\ell} \Bigl( 1 + \frac{y}{n^2} + \frac{r y^2}{n^4} + \cdots \Bigr)  \end{aligned} \\[-0.5ex]
	& = \sum_{\ell\geq1} \prod_{n < \ell} \Bigl( 1 + \frac{x}{n^2 - r x} \Bigr) \cdot \frac{\ell}{(\ell^2 - r x)(\ell^2 - r y)} \cdot \prod_{n>\ell} \Bigl( 1 + \frac{y}{n^2 - r y} \Bigr) \,;
\end{align*}%
the additional factors involving \( \ell \) arise from contracting the index 3 with the indices 2 appearing before it, and/or after it. \medskip

The general principle is that, as long as the repeating blocks have length $1$, the generating function can be expressed by (single or so-called multiple) hypergeometric series.

\subsection{Creative telescoping}\label{sec:ct} We now give a brief overview of the creative telescoping approach, and the algorithmic implementation in Christoph Koutschan's \textsf{HolonomicFunctions} package \cite{Koutschan09}, available for \textsf{Mathematica}. \medskip

Denote by $\textbf{S}_z$ the shift in \( z \) operator: $\textbf{S}_z f(z) \coloneqq f(z+1)$. Let $F(a_1,a_2,\ldots,a_m,n)$ be a term which is hypergeometric (more generally, holonomic with respect to shift operator) in each variable $a_i$ and $n$. The algorithm of creative telescoping produces Ore polynomials (polynomials in $\textbf{S}_{a_i}, \textbf{S}_{n}$ over the field of rational functions in \( a_i, n \)) $\mathbf{Q}_1,\mathbf{Q}_2$  such that 
\[
\mathbf{Q}_1 + (\mathbf{S}_n-1) \mathbf{Q}_2 
\]
 annihilates $F$ and such that $\mathbf{Q}_1$ does not involve $n$. It follows, formally, that
\begin{equation}\label{eqn:ctlimit}
\mathbf{Q}_1 \sum_{n=0}^\infty F(a_1,a_2,\ldots,a_m,n) + \lim_{n\to\infty} \mathbf{Q}_2 F(a_1,a_2,\ldots,a_m,n) - \mathbf{Q}_2 F(a_1,a_2,\ldots,a_m,0) = 0
\end{equation}
so $\sum_{n=0}^\infty F(a_1,a_2,\ldots,a_m,n)$ satisfies a (possibly inhomogenous) recurrence relation whose coefficient are rational functions of the $a_i$. \medskip

The \textsf{Mathematica} package \textsf{HolonomicFunctions} by Christoph Koutschan \cite{Koutschan09} can be used to perform these operations: for a given $F$, the $\mathbf{Q}_1, \mathbf{Q}_2$ can be found using the \texttt{CreativeTelescoping} routine, and the inhomogenous recurrence of $\sum_{n=0}^\infty F(a_1,a_2,\ldots,a_m,n)$ in each \( \mathbf{S}_{a_i} \)can be directly found using \texttt{Annihilator}.  As \texttt{CreativeTelescoping} applies in a quite general setting, we do have to specify that we are interested in the forward difference operation \( \mathbf{S}_n - 1 \) as part of a call to the routine, as well as with respect to which variable(s) we want to find the recurrence/annihilator.

\begin{example}
Consider
\[
F(x,y,n) = \frac{y^3 \poch{1+x}{n}}{\poch{1-x}{n} (n-x+1) ((n+1)^2-y^2)} \,,
\]
where \( \poch{x}{n} \coloneqq x(x+1) \cdots (x+n-1) \) denotes the ascending Pochhammer symbol.  If we want to find a recurrence in \( x \) satisfied by the sum \( \sum_{n=0}^\infty F(x,y,n) \), we call the \texttt{CreativeTelescoping} routine on \( F(x,y,n) \), specify the forward difference operator \( \mathbf{S}_n - 1 \), and ask for the recurrence with respect to \( x \) as follows.
\begin{lstlisting}[]
 <@\inpr@> CreativeTelescoping[ y^3 * Pochhammer[1+x, n] / Pochhammer[1-x, n] / 
            (n+1-x) / ((n+1)^2-y^2), S[n]-1, {S[x]}] // Factor
\end{lstlisting}
Inputting and evaluating this returns the following list 
\begin{mathout} 
\tag*{\outpr}\hspace{-0.5em}\Bigl\{ \{ x (1+x)S_x-(x-y) (x+y) \}, \Bigl\{ -\frac{(-1-n+x) (1+n-y) (1+n+y)}{2 x} \Bigr\} \Bigr\} &&
\end{mathout}%
The first element of this list gives \( \mathbf{Q}_1 \), and the second element gives \( \mathbf{Q}_2 \), i.e. 
\begin{align*}
\mathbf{Q}_1 &= x (1+x)\mathbf{S}_x-(x-y) (x+y) \,,\\
\mathbf{Q}_2 &= -\frac{(-1-n+x) (1+n-y) (1+n+y)}{2 x} \,,
\end{align*}
and indeed \( \mathbf{Q}_1 + (\mathbf{S}_n-1) \mathbf{Q}_2  \) annihilates \( F \). \medskip

If we want to find a recurrence with respect to \( y \), we change  \texttt{S[x]} to \texttt{S[y]}, and make the following call instead.
\begin{lstlisting}[]
<@\inpr@> CreativeTelescoping[ y^3 * Pochhammer[1+x, n] / Pochhammer[1-x, n] / 
            (n+1-x) / ((n+1)^2-y^2), S[n]-1, {S[y]}] // Factor
\end{lstlisting}
This returns
\begin{mathout}
	\tag*{\outpr}\hspace{-0.5em}\Bigl\{ \{ -y^2 (1-x+y)S_y+(1+y)^2 (x+y) \} , \Bigl\{ -\frac{(1+n-x) (1+y)^2 (-1-n+y) (1+2 y)}{2 y (-n+y)} \Bigr\} \Bigr\} &&
\end{mathout}%
This time we read off that
\begin{align*}
	\mathbf{Q}_1 &= - y^2 (1-x+y)\mathbf{S}_y + (1+y)^2 (x+y) \,,\\
	\mathbf{Q}_2 &= -\frac{(1+n-x) (1+y)^2 (-1-n+y) (1+2 y)}{2 y (-n+y)} \,,
\end{align*}
and indeed \( \mathbf{Q}_1 + (\mathbf{S}_n-1) \mathbf{Q}_2  \) annihilates \( F \), now using the shift in \( y \) operator \(\mathbf{S}_y \). \medskip

If one is not interested in seeing the certificate of creative telescoping, $\mathbf{Q}_2$, one can directly input
\begin{lstlisting}[language=]
<@\inpr@> Annihilator[ Sum[ y^3 * Pochhammer[1+x, n] / Pochhammer[1-x, n] / 
          (n+1-x) / ((n+1)^2-y^2), {n, 0, Infinity}], 
        S[y], Inhomogeneous -> True] // Factor
\end{lstlisting}
This returns a list whose first element gives \( \mathbf{Q}_1 \), and whose second element gives an explicit, unevaluated form (wrapped in \texttt{Hold}) of the 
boundary term \( \lim_{n\to\infty} \mathbf{Q}_2 F(a_1,a_2,\ldots,a_m,n) - \mathbf{Q}_2 F(a_1,a_2,\ldots,a_m,0) \) appearing \autoref{eqn:ctlimit}.  Evaluating this command gives the following, where we have manually suppressed the full expression inside the limit\footnote{We have used similar notation \guillemotleft$\cdots$\guillemotright~to that employed by \textsf{Mathematica} with the \texttt{Short} command.}.
\begin{mathout} 
	\tag*{\outpr}\hspace{-0.5em}\Bigl\{ \{ -y^2 (1-x+y)S_y+(1+y)^2 (x+y) \} , \Bigl\{ 
	\frac{1}{2} \bigl(-y-3 y^2-2 y^3+2 \, \mathtt{Hold}\bigl[ \lim_{n\to\infty} \text{\guillemotleft$\cdots$\guillemotright} \bigr]\bigr) \Bigr\} \Bigr\} &&
\end{mathout}%
Again we find the recurrence relation in \( y \) is given via
\[
	\mathbf{Q}_1 = - y^2 (1-x+y)\mathbf{S}_y + (1+y)^2 (x+y) 
\] and one can determine the inhomogeneous part of the recurrence by evaluating the limit\footnote{when finding the limit, one may need to impose certain conditions on the variables from context, such condition usually reflects the convergence of $\sum_{n\geq 0} F(a_1,\cdots,a_m,n)$}.
\end{example}

Creative telescoping gives us a method to determine the recurrence relation satisfied by certain generating series of interest.  We can also directly find general solutions to these recurrences using different techniques.  After doing so we can then explicitly evaluate the genereating series by using analytic properties, and some initial conditions, to determine the unknown coefficients (which will be a 1-periodic functions in variable of the recurrence) which appear in the general solution.  

\subsection{Moderate growth along the imaginary direction} \label{sec:analysis} We now discuss some of the analytic properties we will need, in order to constrain and fix the 1-periodic coefficient functions in the  general solutions to the creative telescoping recurrence relations. \medskip

We start with an elementary observation.
\begin{lemma}\label{cot_poly}
Let $f(z)$ be a meromorphic function of period 1, that is $f(z+1) = f(z)$, and assume it has poles only at $\mathbb{Z}$. Let $0\leq C<\pi$, and suppose $f(z)$ also satisfies the following growth condition:
\begin{equation}\label{mod_growth}|f(x+\ii y)| = O(e^{C |y|}) \qquad y\to \infty\end{equation}
uniformly for $x$ in compact subset of $\mathbb{R}$. Then $f(z)$ is a polynomial in $\cot \pi z$.
\end{lemma}
\begin{proof}
Let $k$ be the order of pole of $f(z)$ at $z=0$. Since $\cot\pi z$ has a simple pole, with residue $1$ at $z=0$, we can choose a degree $k$ polynomial $P$ such that $h(z) = f(z) - P(\cot \pi z)$ is analytic at $z=0$. Because cotangent grows slowly along imaginary direction, we see $h(x+\ii y)$ satisfies the same growth condition in \autoref{mod_growth} as $y\to \infty$ as above, and $h$ is entire by periodicity. Carlson's Theorem \cite[Section~5.3]{bailey1935generalized} implies $h=0$, so $f(z) = P(\cot \pi z)$, as claimed.
\end{proof}

If we generalise the condition that $f$ has poles only at $\mathbb{Z}$, by allowing poles at $a_1 + \mathbb{Z},\ldots, a_r + \mathbb{Z}$, and easy modification of the above proof shows $f(z)$ is a sum of $r$ polynomials, with variables $\cot \pi(z-a_1),\ldots, \cot \pi(z-a_r)$, respectively. 

In this article, a meromorphic function is said to have \textit{moderate growth} if \autoref{mod_growth} is satisfied. A typical example would be the polygamma function $\psi^{(n)}(z) \coloneqq \frac{\dd^{n+1}}{\dd^{n+1} z} \Gamma(z) $. Another example is $\frac{\Gamma(a+z)}{\Gamma(b+z)}$ for $a,b\in \mathbb{C}$; more generally, we have the following.
\begin{lemma}\label{modgrowth_lemma}
Fix $a_0,b_0\in \mathbb{R}$, then we have  \[
\frac{\Gamma(a_0+z)}{\Gamma(b_0+z)} = O(|z|^{a_0-b_0}) \,, \qquad \text{uniformly for } \Re(t)\in \mathbb{R}, \Im(t)>1\,.
\]
\end{lemma}
\begin{proof}
For $0\leq m\leq 1$, this follows from the well-known (e.g. \cite[Equation (5.11.12)]{NIST:DLMF}) asymptotic of gamma function: 
\[
\frac{\Gamma(z+a_0)}{\Gamma(z+b_0)} \sim z^{a_0-b_0} \,, \qquad \text{as $z \to \infty$} \,,
\]
but we still need to show this holds uniformly for all $m\in \mathbb{R}$. We can assume $m>0$, as the case $m<0$ will then follow from the reflection formula of the $\Gamma$ function.  Write \( m = m_0 + N \), with $0\leq m_0\leq 1$, and $N$ a large positive integer, then we have
\[
\frac{\Gamma(a_0 + m_0 + \ii t + N)}{\Gamma(b_0 + m_0 +\ii t+N)} = \biggl( \prod_{k=0}^{N-1} \frac{(a_0+k+z)}{(b_0+k+z)} \biggr)  \cdot \, \frac{\Gamma(a_0+z)}{\Gamma(b_0+z)} \,, \qquad z = m_0 + \ii t \,.
\]
The term $\frac{\Gamma(a_0+z)}{\Gamma(b_0+z)}$ is uniformly bounded, along the imaginary direction, when $0\leq m_0 \leq 1$, so it suffices to prove that the product has the required asymptotics. By taking logarithm and exponential, the absolute value of $\prod_{k=0}^{N-1} \frac{(a_0+k+z)}{(b_0+k+z)}$ is 
\[
\exp\biggl[ \, \Re \sum_{k=0}^{N-1} \log\biggl(1+\frac{a_0}{k+z}\biggl) {} - {} \Re \sum_{k=0}^{N-1} \log\biggl(1+\frac{a_0}{k+z}\biggl) \biggl] = \exp\biggl( \, \sum_{k=0}^{N-1} \log\biggl| 1-\frac{a_0-b_0}{k+z+b_0} \biggr| \, \biggr) \,.
\]
Since $|k+z+b_0|\geq 1$ by our assumption, we can use the Taylor expansion $\log(1-x) = x + O(x^2)$ to conclude the above is 
\[
= O\left(\sum_{k=0}^{N-1} \frac{a_0-b_0}{k+z+b_0}\right) \,,
\]
then the elementary estimation $\sum_{k=1}^N \frac{1}{k} = \log N + O(1)$ gives the claim. 
\end{proof}

A less trivial moderate growth example, which is relevant to us, is the hypergeometric function with respect to its parameters.

\begin{example}
Fix $a,b,c,d,e\in \mathbb{C}$ with $\Re(d+e-a-b-c)>0$, the hypergeometric function
$$\pFq{3}{2}{a+z,b,c}{d+z,e}{1}$$ is a meromorphic function of $z \in \mathbb{C}$. The above lemma implies it is of moderate growth in vertical direction in the variable $z$. To see this, we write above $_3F_2$ as 
$$\frac{\Gamma(e)\Gamma(d+z)}{\Gamma(a+z)\Gamma(b)\Gamma(c)}\sum_{n\geq 0} \frac{\Gamma(a+z+n)\Gamma(b+n)\Gamma(c+n)}{\Gamma(1+n)\Gamma(d+z+n)\Gamma(e+n)}$$
Since $\tfrac{\Gamma(d+z)}{\Gamma(a+z)}$ is of moderate growth, and $\tfrac{\Gamma(a+n+z)}{\Gamma(d+n+z)} = O(|z+n|^{a-d})$ uniformly for $n\geq 0$, this is indeed of moderate growth. A similar assertion also holds, say, for
\[
\pFq{3}{2}{a+z,b-z,c}{d+z,e-z}{1} \,.
\]
\end{example}

\section{Examples of creative telescoping applied to known MZV evaluations}

Now we come to the first main part of this note: we apply the ideas of creative telescoping to a number of known evaluations.  By revisiting known examples, we provide a different viewpoint on how to re-derive the results, which might be useful elsewhere.  

\subsection{Examples \texorpdfstring{$\zeta^\star(\{1\}^m,\{2\}^{n+1})$}{zeta\textasciicircum{}\textasteriskcentered(\{1\}\textasciicircum{}m, \{2\}\textasciicircum{}{n+1}\})}, and \texorpdfstring{$\zeta^\half(\{1\}^a, \mathrm{odd})$}{zeta\textasciicircum{}½(\{1\}\textasciicircum{}a, odd)}}
\label{sec:zs111222}

In Theorem 1.4 of \cite{ChenEie}, the authors applied ideas from Yamamoto's 2-coloured poset integrals \cite{YamamotoPoset17} to establish the following evaluation,
\[
	\zeta^\star(\{1\}^m, \{2\}^{n+1}) = \sum_{p+r = n} (-1)^r \zeta^\star(\{2\}^p) \cdot \sum_{\abs{\vec{d}} = m} \zeta(d_1 + 2, \ldots, d_{r+1} + 2) \prod_{j=1}^{r+1} (d_j+1) \,.
\]
It follows that \( \zeta^\star(\{1\}^{m},\{2\}^{n+1}) \) is a polynomial in single zeta values for any \( n,m \geq 0 \in \mathbb{Z} \):  
by the symmetric sum theorem \cite{hoffman92}, the inner sum over all compositions of \( m \) into \( d \) parts can be expressed as a polynomial in Riemann zeta values.  (In \S6.1 \cite{ChenEie}, the special cases \( n = 0, 1 \) are treated, in \S6.2 \cite{ChenEie}, so are \( m = 1, 2, 3 \); handling the general case via the symmetric sum theorem does not seem to be explicitly mentioned.)

Extracting a closed formula for particular cases of fixed \( m \) or fixed \( n \), by directly applying the symmetric sum theorem, is easy enough for small values \( n = 0, 1 \) or \( m = 1, 2, 3 \) (c.f. \S6.1, \S6.2 \cite{ChenEie}); in general it seems not too enlightening and can be error-prone. Here we outline a method that directly yields the generating function of $\zeta^\star(\{1\}^m, \{2\}^{n+1})$.  Readers interested in computational details for this example should consult the \textsf{Mathematica} notebook attached.\medskip

As a starting point, we apply the generalised two-one \cite{ZhaoIdentity16} formula, which gives us that for \( m \geq 1 \), 
\[
	\zeta^\star(\{1\}^m, \{2\}^{n+1}) = 2^{m} \zeta^\half(\{1\}^{m-1}, 2n+3) \,.
\]
Note, however, that we shall have to excluded the terms with \( m = 0 \),\vspace{-0.5em}
\[
	\zeta^\star(\{2\}^{n+1}) = \frac{(-1)^n 2(2^{2n+1} - 1) B_{2n+2} \pi^{2n+2}}{(2n+2)!}  = -2 \zeta^\half(\overline{2n+2}) \,,
\]
but this evaluation is well-known \cite[Equation~36 and thereafter]{HoffmanIhara17} or \cite[p. 203]{IharaMZSV}, and can be incorporated easily into the final result, if desired.

So consider the following generating series, where \( v = \frac{2x}{\ell} \), \( u = \frac{2x}{n} \) inside the sum,
\begin{align*}
	& \sum_{m,n \geq 0} \zeta^r(\{1\}^{m}, 2n+3) \cdot (2x)^{m} y^{2n+3} \\
	&= \sum_{\ell\geq1} \sum_{n\geq0} \frac{y^{2n+3}}{\ell^{2n+3}} (1 + r + v^2 r^2 + \cdots) \cdot  \prod_{k < \ell} (1 + u + u^2 t + u^3 t^2 + \cdots ) \\
	&= \sum_{\ell\geq1} \sum_{n\geq0} \frac{y^{2n+3}}{\ell^{2n+3}} \cdot \frac{1}{1- v r} \cdot  \prod_{n<\ell} \Big(1 + \frac{u}{1-u r} \Big)
\end{align*}
Specialise to \( r = \half \), and after some straightforward calculations, we find
\[
f(x,y) \coloneqq \sum_{m,n\geq0} \zeta^\half(\{1\}^m, 2n+3) (2x)^m y^{2n+3} = \sum_{k\geq0} \frac{y^3 \poch{1+x}{k}}{((k+1)-x)((k+1)^2 - y^2) \poch{1-x}{k}} \,.
\]
The \texttt{Annihilator} (or \texttt{CreativeTelescoping}) routine in \textsf{HolonomicFunctions} gives us the relation (valid for $x<1$):
\[
-y^2  (-x+y+1) f(x,y+1)+(y+1)^2  (x+y) f(x,y)+\frac{-2 y^3-3 y^2-y}{2}=0 \,.
\]
The general solution to above recurrence in $y$, is
\begin{align*}
	f(x,y) = \frac{y}{2x} + \frac{\pi y^2(1 - 2 x \cdot c) \csc{\pi x} \, \Gamma(x+y)}{2x^2 \Gamma(x - 1) \Gamma(x) \Gamma(1 - x + y)} \,.
\end{align*} Here $c = c(x,y) $ is a ``constant'' in \( y \), in the sense that is unchanged under $y\to y+1$, i.e. $c(x,y)$ is 1-periodic function in $y$\footnote{the recurrence can be solved using the built-in \texttt{RSolve} command of \texttt{Mathematica} and simplification afterwards}. 
For generic fixed $x$, $f(x,y)$ has simple poles at $y\in \mathbb{Z}$, and it is of moderate growth in $y$ along imaginary direction, so by  \autoref{cot_poly}, 
\[
c(x,y) = A(x) + B(x) \cot{\pi y} \,,
\] for some meromorphic $A(x), B(x)$ depending only on $x$.

The functions $A(x), B(x)$ can be fixed by considering the coefficients of $y^1, y^2$, which must vanish identically since $f(x,y)$ starts with order $y^3$.  Solving the resulting system of equations tells us
\[
\begin{cases}
	\, A(x) = \frac{1 - x - \pi  x \cot{\pi  x}}{2(1 - x)x} \,, \qquad B(x) = -\frac{\pi }{2 (1-x)} \,.
\end{cases}
\]
Substituting these back into \( c(x,y) \), and simplifying, finally gives us our desired generating series
	\[
		f(x,y) = \frac{y}{2x} + \frac{\pi^2 x y^2 (1 - \cot{\pi x} - \cot{\pi y}) \csc{\pi x}  \tan{\pi(x+y)}}{2(x+y)} \cdot \frac{\Gamma(1 + x + y)}{\Gamma(1+x)^2 \Gamma(1 - x + y)} \,.
	\]

In general, we can then extract and calculate (for $m\geq1,n\geq0$) that
\[
 \zeta^\star(\{1\}^m, \{2\}^{n+1}) = 2 \cdot [x^{m-1} y^{2(n+1)+1}] \, f(x,y) \,.
\]
Using the Taylor expansion \( \log\Gamma(1+z) = -\gamma z + \sum_{k=2}^\infty \frac{(-z)^k}{k} \zeta(k) \), we hence conclude that both \( \zeta^\star(\{1\}^{m},\{2\}^{n+1}) \) and \( \zeta^\half(\{1\}^{m-1}, 2n+3) \) are polynomials in Riemann zeta values.  For example (after some work) one obtains,
{\small
\begin{align*}
& \zeta^\star(\{1\}^5,\{2\}^4) = 2 \cdot [x^{5-1} y^{2\cdot4 + 1}] \, f(x,y) \\
& {
\text{ \small $ \displaystyle -\frac{2\pi ^4}{45} \zeta (3)^3+20 \zeta (7) \zeta (3)^2+16 \zeta (5)^2 \zeta (3)-\frac{17 \pi ^{10}}{31185}  \zeta (3)-\frac{\pi ^8 }{180}\zeta (5)-\frac{4 \pi ^6}{63}\zeta (7)-\frac{28 \pi ^4}{45}\zeta (9)+198 \zeta (13) $} \,.
}
\end{align*}%
}
 
\subsection{Example \texorpdfstring{$\zeta(\{2\}^n,3,\{2\}^m)$}{zeta(\{2\}\textasciicircum{}n, 3, \{2\}\textasciicircum{}m)}}
\label{sec:z2232}

We show how Zagier's celebrated result (\cite{zagier2232}) evaluating the MZV  $\zeta(\{2\}^m,3,\{2\}^n)$ can be (somewhat directly) re-derived using the creative telescoping methods. Readers interested in computational details for this example should consult the \textsf{Mathematica} notebook attached.\medskip 

First we write down the generating function (very similar to Proposition 1 in \cite{zagier2232}):
\[
h(x,y) \coloneqq \sum_{n,m\geq 0} \zeta(\{2\}^m,3,\{2\}^n) (-x^2)^n (-y^2)^m = \sum_{\ell\geq 1}\frac{1}{\ell^3} \prod_{n<\ell} \biggl(1-\frac{x^2}{n^2}\biggr) \prod_{n>\ell}\biggl(1-\frac{y^2}{n^2}\biggr) \,.
\]
In terms of hypergeometric functions, this is\footnote{It can also be written as a derivative of ${}_3F_2$ with respect to a parameter, as per \cite{zagier2232}.}
\[
 = \frac{1}{\Gamma(2-y)\Gamma(2+y)}\pFq{4}{3}{1,1,1-x,1+x}{2,2-y,2+y}{1} \,.
\]
Our goal is to explicitly find above ${}_4F_3$ in terms of polygamma functions, this will give an explicit reduction of $\zeta(\{2\}^a,3,\{2\}^b)$ into single zeta values.  \medskip

It turns out that it is more convenient to work with the following modified version\footnote{After the relabelling of $x,y$, the resulting function is of moderate growth w.r.t. each variable by \autoref{modgrowth_lemma}, whereas the original version is not.}  of \( h(x,y) \) obtained by replacing $(x,y)$ by $(x-y,x+y)$, and plus a rational factor.  So consider
\begin{align*}
f(x,y) & \coloneqq \sum_{m\geq 0} \underbrace{ \frac{(y^2-1) (x-y)^2}{(x+y)^2-1} \cdot \frac{\poch{1}{m} \poch{1-x+y}{m} \poch{1+x-y}{m}}{\poch{2}{m} \poch{2-x-y}{m} \poch{2+x+y}{m}}}_{\eqqcolon F(x,y,m)} \\
 & = \frac{(x-y)^2 (y^2 - 1)}{(x+y)^2 - 1} \pFq{4}{3}{1,1,1-x+y,1+x-y}{2,2-x-y,2+x+y}{1} \,.
\end{align*}

Using \texttt{Annihilator} (or \texttt{CreativeTelescoping}) in \textsf{HolonomicFunctions}, plus some further human-guided simplification afterwards, we find that $f(x,y)$ satisfies the recurrence
\[
f(x,y)-f(x+1,y)  = (y^2-1) \Big( \Big( \frac{1}{4x} - \frac{1}{x+y} - \frac{1}{1+x+y} + \frac{1}{1+2x} \Big) + \frac{\csc\pi(x+y) \sin\pi(x-y)}{4x(1+x)(1+2x)} \Big) \,.
\]

The above can be viewed as a first order recurrence in $x$, and one can verify that
\begin{align*}
	g(x,y) = {} & \frac{y^2 - 1}{4 x y} \biggl(  \frac{(x-y)^2}{x+y} - (x-y) \csc{\pi(x+y)} \sin{\pi(x-y)} \biggr) \\
	& \begin{aligned}[t] {} + (y^2 -  1) {} & \Bigl\{ 2 \psi(1 + x + y) - \csc{\pi(x+y)} \sin{\pi(x+y)}  \cdot ( \psi(1+2x) + \psi(1 + 2y) )  \\
	& - \csc{\pi(x+y)} \sin{\pi(x-y)} \cdot \big( \psi(1+x) - \psi(1 + 2x) - \psi(1 + y) + \psi(1 + 2y) \big) \Bigr\} \end{aligned}
\end{align*}
also satisfies this recurrence.   (This is easily checked, but beware that \texttt{RSolve} seems to be unreliable and produces incorrect -- but closely related results -- in this case.)

Therefore $f(x,y) = g(x,y) + c(x,y)$ where $c(x,y)$ is $1$-periodic in $x$. By looking at the definitions of \( f \) and \( g \), we see that for generic $y$, $f(x,y)$ and $g(x,y)$ could have poles only at $x\in -y +  \mathbb{Z}$. Note $f$ is analytic at $x=-y$; it can be checked easily that $g(x,y)$ is also analytic at $x=-y$.  Therefore $c(x,y)$ is analytic at this point, thus by periodicity, it is an entire function of $x$.  Being of moderate growth in $x$, we see $c(x,y)$ depends only on $y$, so we have
\[
c(x,y) = c(0,y) = f(0,y) - g(0,y) \,.
\]
However, $f(0,y)$ can be easily calculated, because $F(x,0,m)$ reduces a rational function in $x,m$.  We find
\[
f(0,y) = \frac{1}{2} (y^2 - 1) \big( 2\gamma - y^{-1} + \pi \cot{\pi y} + 2 \psi(1+y) \big)\,,
\]
and this agrees with \( g(x,y) \) as \( x \to 0 \).  Therefore $c(x,y)$ is identically $0$, so we have proved $f(x,y) = g(x,y)$

\begin{remark} If one has not so presciently guessed the form of \( g(x,y) \), one would take any other solution to the recurrence, and find by \autoref{cot_poly}, that for generic \( y \), \( c(x,y) = A(y) + B(y) \cot{\pi(x+y)} \), as \( f(x,y) - g(x,y) \) is $1$-periodic, of moderate growth in $x$, with at most a simple pole at \( x \in -y + \mathbb{Z} \).  Analyticity of \( f(x,y) \) at \( x = -y \), and considering \( x = 0 \) would lead to equations for \( A(y), B(y) \), which, when solved, again prove the evaluation of \( f(x,y) \).
\end{remark}

We then recover
\[
	h(x+y,x-y) = \frac{((x+y)^2 - 1) \cdot f(x,y)}{\Gamma(2-x-y)\Gamma(2+x+y)(x-y)^2 (y^2 - 1)}  = -\frac{\sin{\pi(x+y)} \cdot f(x,y)}{\pi(x-y)^2 (x+y)(y^2 - 1)}  \,,
\]
or by replacing \( (x,y) \) by \( (\tfrac{x+y}{2}, \tfrac{-x+y}{2}) \), we directly have
\[
	h(x,y) = \frac{4 \sin{\pi y}}{\pi x^2 (4 - (x-y)^2)} f\Bigl( \frac{x+y}{2}, \frac{-x+y}{2} \Bigr) \,.
\]
Using the Taylor series expansion \( \psi(1+z) = -\gamma + \sum_{k\geq 2} \zeta(k) (-z)^{k-1} \), we therefore see the MZV's $\zeta(\{2\}^n,3,\{2\}^m)$ can indeed be expressed via Riemann zeta values.

\subsection{Example \texorpdfstring{$t(\{2\}^n,3,\{2\}^m)$}{t(\{2\}\textasciicircum{}n, 3, \{2\}\textasciicircum{}m\})}}
\label{sec:t2232}

Murakami \cite{murakami21} showed that $t(\{2\}^n,3,\{2\}^m)$ can be expressed in terms of Riemann zeta values, and from this deduced \( \{ t^\mathfrak{m}(k_1,\ldots,k_d) \mid k_i = 2, 3 \} \) form a basis of the space of motivic multiple \emph{zeta} values (as well as showing generally any MtV with all indices \( {>}1 \) enjoys a Galois descent to MZV's).  We give an alternative proof of Murakami's evaluation, using similar manipulations as above.  Readers interested in computational details for this example should consult the \textsf{Mathematica} notebook attached\footnote{we remark that the formulas involved in this example on $t(\{2\}^m,3,\{2\}^n)$ are much shorter than those involved in $\zeta(\{2\}^m,3,\{2\}^n)$ above.}.  \medskip

Consider the following generating series of the MtV's
\begin{align*}
\sum_{n,m\geq 0}   t(\{2\}^n,3,\{2\}^m) & {} \cdot 2^3 (-4x^2)^n (-4y^2)^m \\[-1.5ex]
&= \sum_{\ell\geq 0} \frac{1}{(\ell+\tfrac{1}{2})^3} \prod_{0\leq n<\ell} \bigg(1-\frac{x^2}{(n+\tfrac{1}{2})^2}\bigg) \prod_{n>\ell} \bigg( 1-\frac{y^2}{(n+\tfrac{1}{2})^2}\bigg)
\\
 &= \sum_{\ell\geq 0} \frac{2 \cos \pi  x\, \Gamma \big(\ell-x+\tfrac{1}{2}\big) \Gamma \big(\ell+x+\tfrac{1}{2}\big)}{(2 \ell+1) \Gamma \big(\ell-y+\tfrac{3}{2}\big) \Gamma \big(\ell+y+\tfrac{3}{2}\big)} \,. 
 \end{align*}
As before, it is more convenient to work with the following version, where  $(x,y) $ is replaced by $(x-y,x+y)$.  So our goal is to express
\begin{equation}\label{eqn:sumt2232}
f(x,y) \coloneqq \sum_{k\geq 0} \frac{2 \cos \pi  (x-y) \, \Gamma \big(k+x-y+\frac{1}{2}\big) \Gamma \big(k-x+y+\frac{1}{2}\big)}{(2 k+1) \, \Gamma \big(k-x-y+\frac{3}{2}\big) \Gamma \big(k+x+y+\frac{3}{2}\big)}
\end{equation}
 in terms of polygamma functions. 

Using \textsf{HolonomicFunctions}, we find that $f(x,y)$ satisfies the following recurrence in $x$
\begin{align*}
 f(x+1,y) \bigl(-x^2-2 x & {} +y^2-1 \bigr) 
+f(x,y) \bigl( y^2-x^2 \bigr) \\
&  +\frac{(8 x^2+8 x+1) \cos \pi  (x+y) - \cos \pi  (x-y)}{4 x (x+1) (2 x+1)}=0 \,.
\end{align*}
The general solution to this is 
\begin{align*}
	f(x,y) = 
	\begin{aligned}[t] 
	 \frac{\cos \pi(x-y)}{4(x^2 - y^2)} & \bigl( 4(1-y^2) \cdot c + x^{-1} + 3 + 4 \psi(x) - 4 \psi(2x) \bigr) \\[-1ex]
	& + \frac{\cos \pi(x+y)}{4(x^2 -y^2)} \bigl( -x^{-1} + 5 - 4 \gamma - 4 \psi(2x) \bigr) \,.
	\end{aligned}
\end{align*}
where $c(x+1,y) = c(x,y)$.  (This is easily checked, however beware that \texttt{RSolve} seems to be unreliable and produces incorrect -- but closely related -- results also in this case.)  

For generic $y$, $\frac{f(x,y)}{\cos \pi(x-y)}$ is of moderate growth in $x$, it has poles at $x\in y+\big(\mathbb{Z}+\frac{1}{2}\big)$, and the various digamma functions above also produce a simple pole at $x\in \mathbb{Z}+\frac{1}{2}$. Therefore \autoref{cot_poly} tells us
\[
c(x,y) = A(y) + B(y) \tan\pi(x-y) + C(y)\tan (\pi x) \,,
\]
with $A(y), B(y), C(y)$ to be determined. 

To determine \( A(y), B(y), C(y) \), we expand the tentative solution \( f(x,y) \)  at \( x = \pm y, 0 \).  At \( x = \pm y \), it must analytic, so the coefficient of \( (x\pm y)^{-1} \) must vanish identically.  At \( x = 0 \), one can directly evaluate \autoref{eqn:sumt2232}, as the summand becomes a simple rational function (and \( \cos{\pi y} \)), so the tentative solution must also give
\[
	f(0,y) = -\frac{\cos{\pi y}}{y^2} \big( \gamma - \psi(y) + 2 \psi(2y)  \big) + \frac{\pi \sin{\pi y}}{2 y^2} \,.
\]
From these conditions, we find
\[
	\begin{cases}
		A(y) = \displaystyle\frac{1 - \cos{2\pi y}}{4y(y^2 - 1)} + \frac{1}{y^2 - 1} \biggl( \frac{3}{4}  + \frac{5}{4} \cos{2 \pi y} + (\psi(y) + \gamma) - 2 (\psi(2y) + \gamma) \cos^2{\pi y}  \biggr) \\[2ex]
		B(y) = \displaystyle\frac{\sin{2\pi y}}{y^2 - 1} \biggl( \frac{1-5y}{4y} + \gamma + \psi(2y) \biggr) \\[2ex]
		C(y) = 0 \,.
	\end{cases}
\]

Substituting these into the tentative solution, and simplifying, gives the following final closed form solution
\[
	f(x,y) = \begin{aligned}[t]
	\frac{\sin{\pi x}\sin{\pi y}}{2 x y (x + y)} & + \frac{\cos{\pi(x-y)}}{x^2 - y^2} \big( \psi(1 + x) - \psi(1 + y)\big) \\[-0.5ex]
	& - \frac{2\cos{\pi x}\cos{\pi y}}{x^2 - y^2} \big( \psi(1 + 2x) - \psi(1 + 2y) \big)\,. \end{aligned}
\]
The original generating series we wanted, \( h(x,y) \), is then recovered simply as \( h(x,y) = f(\frac{x+y}{2}, \frac{-x+y}{2}) \).  This shows that \( t(\{2\}^n, 3, \{2\}^m) \) can indeed be expressed by Riemann zeta values, by applying the Taylor series expansion \( \psi(1+z) = -\gamma + \sum_{k\geq 2} \zeta(k) (-z)^{k-1} \).\vspace{0.5ex}
\begin{center}
	------------------
\end{center}

We now apply creative telescoping methods to some (apparently) new MZV-type evaluations.  With the new results, we settle some conjectures or questions raised elsewhere.%
\section{Results on \texorpdfstring{$\zeta^\half(\{\overline{2}\}^n,3)$}{zeta\textasciicircum{}½(\{-2\}\textasciicircum{}m,3)}, \texorpdfstring{$\zeta^\half(3,\{\overline{2}\}^n)$}{zeta\textasciicircum{}½(3,\{-2\}\textasciicircum{}m)}, and zeta stars involving \texorpdfstring{\( \{1,3\}^m \)}{\{1,3\}\textasciicircum{}m}}
\label{sec:zh223}

	In this section, we show an evaluation for a combination of \( {}_6F_5 \) hypergeometric series, and use this to give explicit evaluations of \( \zeta^\half(\{\overline{2}\}^n, 3) \), and \( \zeta^\half(3,\{\overline{2}\}^n) \) in terms of Riemann zeta values.  By applying the generalised two-one theorem we convert these into four families of multiple zeta star value evaluations. \medskip

    As a starting point, using the principal of \autoref{sec:genser}, we can write down the following generating series, where $u=\frac{(-1)^n x}{n^2}, v = \frac{(-1)^\ell x}{l^2}$ inside the summation.
	\begin{align*}
	\sum_{k\geq 0}\zeta^r(\{\overline{2}\}^k,3) x^k &= \sum_{\ell\geq 1} \frac{1}{\ell^3} (1+v r+v^2 r^2 + \cdots)\prod_{n<\ell}  (1+u+u^2 r + u^3 r^2 +\cdots ) \\
    &= \sum_{\ell\geq 1} \frac{1}{\ell^3} \frac{1}{1-vr} \,  \prod_{n<\ell} \Bigl(1+\frac{u}{1-ur}\Bigr) \,,
   \\[1ex]
   \sum_{k\geq 0}\zeta^r(3,\{\overline{2}\}^k) x^k &= \sum_{\ell\geq 1} \frac{1}{\ell^3} (1+vr+v^2r^2 + \cdots) \,  \prod_{\ell>n}  (1+u+u^2 r + u^3 r^2 +\cdots ) \\
    &= \sum_{\ell\geq 1} \frac{1}{\ell^3} \frac{1}{1-vr} \, \prod_{n>\ell} \Bigl(1+\frac{u}{1-ur} \Bigr) \,.
    \end{align*}
    It is most convenient to work with the generating series weighted by \( (8x^2)^k \) instead of just \( x^k \); we must also specialise to \( r = \half \).  If we split the resulting sum into \( \ell = 2m+2 \) even, and \( \ell = 2m+1 \) odd terms ($m\geq0$), we can write
    \begin{align*}
    	\sum_{k\geq0} \zeta^\half(\{\overline{2}\}^n, 3) (8x^2)^k = 
    	-\frac{1}{x^2} \sum_{m\geq0} \big( F_\text{even}(x, -\ii x, m) + F_\text{odd}(x, -\ii x, m) \big) \,, \\
    	\sum_{k\geq0} \zeta^\half(3,\{\overline{2}\}^n) (8x^2)^k = 
    	\frac{\cot{\pi x}\tanh{\pi x}}{x^2} \sum_{m\geq0} \big( F_\text{even}(\ii x, x, m) + F_\text{odd}(\ii x, x, m) \big) \,,
    \end{align*}
    where
    \begin{align*}
    	F_\mathrm{even}(x,y,m) & = -\frac{(4x^2-1)y^2}{8(x^2-1)(4y^2-1)} \cdot \frac{\poch{1}{m}^2 \poch{\tfrac{3}{2}-x}{m} \poch{\tfrac{3}{2}+x}{m}\poch{1-y}{m} \poch{1+y}{m} }{\poch{2}{m} \poch{2-x}{} \poch{2+x}{m} \poch{\frac{3}{2}-y}{m} \poch{\frac{3}{2}+y}{m}} \cdot \frac{1}{m!} \,,
   \\
   F_\mathrm{odd}(x,y,m) & = 
   	-\frac{y^2}{4y^2-1} \cdot \frac{\poch{\tfrac{1}{2}}{m} \poch{1}{m} \poch{\tfrac{1}{2}-x}{} \poch{\tfrac{1}{2}+x}{m} \poch{1-y}{m} \poch{1+y}{m}}{\poch{\tfrac{3}{2}}{m} \poch{1-x}{m} \poch{1+x}{m} \poch{\tfrac{3}{2}-y}{m} \poch{\tfrac{3}{2}+y}{m}} \cdot \frac{1}{m!} \,,
   \end{align*}
   and the factor \( \cot(\pi x) \tanh(\pi x) \) comes from  \[
   \cot(\pi x) \tanh(\pi x) = \prod_{n=1}^\infty \Bigl( 1 + \frac{u}{1 - \frac{1}{2} u}\Bigr) \,, \qquad u = \frac{(-1)^n x}{n^2} \,.
   \]
	Definitionally then, one has
	\begin{align*}
		\sum_{m\geq0} F_\mathrm{even}(x,y,m) &= -\frac{(4x^2-1)y^2}{8(x^2-1)(4y^2-1)} \cdot \pFq{6}{5}{1,1,\tfrac{3}{2}-x,\tfrac{3}{2}+x,1-y,1+y}{2,2-x,2+x,\tfrac{3}{2}-y,\tfrac{3}{2}+y}{1} \,, \\
		\sum_{m\geq0} F_\mathrm{odd}(x,y,m) &= -\frac{y^2}{4y^2-1} \cdot  \pFq{6}{5}{\tfrac{1}{2},1,\tfrac{1}{2}-x, \tfrac{1}{2}+x, 1-y, 1+y}{\tfrac{3}{2}, 1-x, 1+x, \tfrac{3}{2}-y,\tfrac{3}{2}+y}{1} \,.
	\end{align*}

	On account of the following theorem, we can evaluate this combination of \( {}_6F_5 \) functions through digamma and trigonometric functions.

\begin{theorem}\label{theorem:6F5identity}
For all $x,y\in \mathbb{C}$, the combination of \( {}_6F_5 \) functions
\[
\sum_{m\geq 0} \big( F_\mathrm{even}(x,y,m) + F_\mathrm{odd}(x,y,m) \big)
\] 
is equal to the following 
\begin{align*}
& \frac{1}{8} \psi (1+x+y)
+\frac{1}{8} \psi(1+x-y)
-\frac{1}{4}\psi(1+x)
+\frac{1}{4} \psi(1+y)
-\frac{1}{2} \psi(1+2 y)
+\frac{1}{8 x}
-\frac{\gamma }{4}
+\frac{\pi}{8} \tan{\pi y}
\\ 
& 
+\frac{\cot{\pi  x} \tan {\pi  y}-1}{16 (x-y)}
-\frac{\cot{\pi  x} \tan {\pi y}+1}{16 (x+y)}
+\frac{1}{8} \cot{\pi  x} \tan{\pi  y} \cdot  \big(\psi(1+x+y)-\psi(1+x-y)\big)
\,. \end{align*}
\end{theorem}

Using the Taylor series expansion $\psi(1+z) = -\gamma + \sum_{k\geq 2} \zeta(k) (-z)^{k-1}$, we directly obtain the following explicit evaluations of $\zeta^\half(\{\overline{2}\}^n,3)$ and $\zeta^\half(3,\{\overline{2}\}^n)$.

\begin{corollary} \label{cor:zh2223}
	We have the following explicit reduction of $\zeta^{\half}(\{\overline{2}\}^n,3)$ and $\zeta^{\half}(3,\{\overline{2}\}^n)$ into Riemann's zeta, with the convention that \( \zeta(0) = -\frac{1}{2} \):
{\[
\begin{aligned}
&\zeta^{\half}(\{\overline{2}\}^n,3) =
\begin{aligned}[t]
  \Big(-\frac{\delta_{\text{$n\,$even}}}{2 \cdot 8^n} + {} & {} \frac{\delta_{\text{$n\,$odd}}}{2 \cdot (-4\ii)^n} + \frac{2}{(-2)^n} \Big) \zeta(2n+3) \\
 & + \frac{8}{3\zeta(2)} \sum_{\substack{a + b + 2c = n+1 \\ a,b,c \geq 0}} \!\!\! \frac{(-1)^{a+c}(1 - 2^{2a})}{2^{3a + 3b + 4c}} \zeta(2a)\zeta(2b) \zeta(4c+3) \,,
 \end{aligned} \\
&\zeta^{\half}(3,\{\overline{2}\}^n) = \begin{aligned}[t]
& -\frac{\ii^n \, \delta_{\text{$n\,$even}}}{2^{2n+1}} \zeta(2n+3) \\
&  + \frac{1}{6\zeta(2)} \sum_{\substack{a + b + c = n + 1 \\ a,b,c\geq0}} \!\!\! \begin{aligned}[t] 
 \frac{(-1)^a (1 -2^{2a})}{8^n} & \zeta(2a)\zeta(2b) \zeta(2c+3) \\[-1ex]
&
\cdot \big( 2 \,\delta_{\text{$c\,$even}}+ (2\ii)^{c+1} \,\delta_{\text{$c\,$odd}}  - 8 \cdot 4^c  \big) \,. \end{aligned} 
\end{aligned}
\end{aligned}
\]
}
Here we extend the Kronecker delta so that \( \delta_\bullet = 1 \) if \( \bullet \) is true, and \( \delta_\bullet = 0 \) otherwise.\footnote{Note: Since the factor \( 1 - 2^{2a} = 0 \) for \( a = 0 \), each term of the sum necessarily contains a factor of \( \pi^2 \), which can be cancelled against the \( \zeta(2) \) in the denominator.  Although seeing division by \( \zeta(2) \) in evaluations is somewhat unusual, this form seems to be the best way to write the results. Alternatively one could simplify \( \frac{\zeta(2a)}{\zeta(2)} \) using the known evaluation, at the expense of introducing Bernoulli numbers and factorials.}
\end{corollary}

By the generalised two-one theorem \cite{ZhaoIdentity16} (c.f. also \cite{CharltonCyclic20}, for a description via the block decomposition, which gives a more direct way to pass from \( \zeta^\star \) to \( \zeta^\half \) arguments), we have
\begin{align}
\label{eqn:zh223aszs1313:a} \zeta^\star(\{1,3\}^n,1,2) &= 2^{1+2n} \zeta^\half(\{\overline{2}\}^{2n},3) \\
\label{eqn:zh223aszs1313:b} \zeta^\star(2,\{1,3\}^n,1,2) &= -2^{2+2n} \zeta^\half(\{\overline{2}\}^{2n+1},3) \\
\zeta^\star(1,2,\{1,3\}^n) &= 2^{1+2n} \zeta^\half(3,\{\overline{2}\}^{2n}) \\
\label{eqn:zh223aszs1313:d} \zeta^\star(2,3,\{1,3\}^n) &= -2^{2+2n} \zeta^\half(3,\{\overline{2}\}^{2n+1}) \,.
\end{align}
Therefore these multiple zeta star values can also be expressed as polynomials in Riemann zeta values.  In particular we have the following explicit evaluations.

\begin{corollary}\label{cor:zs1313}
	The following evaluations hold for all \( n \geq 0 \), with the convention that \( \zeta(0) = -\frac{1}{2} \):
	{\small 
	\begin{align*}
		\zeta^\star(\{1,3\}^n, 1,2) &= \Big(4 - \frac{1}{16^n}\Big) \zeta(4n+3) + \frac{8}{3\zeta(2)} \sum_{\substack{a + b + 2c = 2n+1 \\ a,b,c \geq0}} \!\!\!\!  \frac{(-1)^{a+c} (1 - 2^{2a}) }{4^{a+b+c}} \zeta(2a)\zeta(2b)\zeta(4c+3) \,, \\
		\zeta^\star(2,\{1,3\}^{n},1,2) &= \Big( 4 + \frac{1}{2 \cdot(-4)^n} \Big) \zeta(4n+5) - \frac{8}{3\zeta(2)} \sum_{\substack{a + b + 2c = 2n+2 \\ a, b, c \geq 0}} \!\!\!\! \frac{(-1)^{a+c} (1-2^{2a})}{4^{a+b+c}} \zeta(2a)\zeta(2b)\zeta(4c+3) \,, \\	
\intertext{\normalsize and}
		\zeta^\star(1,2,\{1,3\}^n) &= -\frac{\zeta(4n+3)}{(-4)^n} + \frac{1}{3 \zeta(2)} \sum_{\substack{a + b + c = 2n + 1 \\ a,b,c \geq 0}} \!\!\!\! \begin{aligned}[t]  \frac{(-1)^a (1 - 2^{2a})}{16^{n}} & \zeta(2a)\zeta(2b)\zeta(2c+3)  \\[-1ex]
		 & \cdot (2 \,\delta_{\text{$c\,$even}} + (2\ii)^{c+1}\, \delta_{\text{$c\,$odd}} - 8 \cdot 4^c) \,, \end{aligned} \\[1ex]
		 \zeta^\star(2,3,\{1,3\}^n) &= -\frac{1}{12 \zeta(2)} \sum_{\substack{a + b + c = 2n + 2 \\ a,b,c \geq 0}} \begin{aligned}[t]
		  \frac{(-1)^a (1 - 2^{2a}) }{16^{n}} & \zeta(2a)\zeta(2b)\zeta(2c+3) \\[-1ex]
		 & \cdot (2 \,\delta_{\text{$c\,$even}} + (2\ii)^{c+1}\, \delta_{\text{$c\,$odd}} - 8 \cdot 4^c)  \,. \end{aligned}
	\end{align*}
}
\end{corollary}

	\begin{remark}[Evaluations via standard techniques]\label{remark:mzvtechniques}
		Some parts of these evaluations can be handled with more standard MZV techniques; here is a brief outline.  Using the stuffle-antipode (c.f. \cite[Lemma 4.2.2]{GlanoisThesis}, \cite[Lemma 3.3]{GlanoisUnramified}, \cite{HoffmanQuasi2020}), entailing stuffle-regularisation, we have that
		\begin{align*}
			& \zeta^\star(1,\{1,3\}^n) = \\
			& \sum_{k=0}^n \zeta^\star(\{1,3\}^k) \, \reg_{\ast,T} \zeta(\{3,1\}^{n-k}, 1) - \sum_{k=0}^{n-1} \zeta^\star(3,\{1,3\}^k) \, \reg_{\ast,T} \zeta(1,\{3,1\}^{n-1-k}, 1) \,.
		\end{align*}
		
		From \cite[Theorem 6]{munetaZS13}, we can deduce that \( \zeta^\star(\{1,3\}^k) \) and \( \zeta^\star(3,\{1,3\}^k) \) are polynomials in Riemann zeta values (the former is given explicitly in Theorem B therein, for the latter we need to know that \( \zeta(3,\{1,3\}^k) \) is a polynomial in Riemann zeta values, which is given in Theorem 1 of \cite{bowmanBradleyResolution}).
		
		By considering shuffle regularisation, and the shuffle products of \( \zeta(\{2\}^i) \) with \( \zeta_1(\{2\}^j) \), the dual of \( \zeta(\{2\}^j,1) \), and of  \( \zeta(\{2\}^k,1) \) with itself (c.f. \cite[Theorem 4.1]{bowmanBradleyShuffles}),  we have
		\begin{align}
		\label{eqn:z131} \sum_{i=0}^{2n+1} (-1)^i \reg_{\shuffle,T} \, \zeta(\{2\}^i) \cdot  \zeta_1(\{2\}^{2n+1-i}) &= (-1)^{n} 2^{2n} \reg_{\shuffle,T} \, \zeta(\{3,1\}^{n}, 1) \,, \\
		\label{eqn:z1311} \sum_{i=0}^{2n} (-1)^i \reg_{\shuffle,T} \, \zeta(\{2\}^i,1) \cdot \zeta(\{2\}^{2n-i},1) &= (-1)^{n} 2^{2n+1} \reg_{\shuffle,T} \, \zeta(\{1,3\}^{n},1, 1) \,.
		\end{align} 
		As \( \zeta(\{2\}^j,1) \) is a polynomial in single-zetas (c.f. \cite[Lemma 3.8]{BrownMTM12}), thence \(  \zeta_1(\{2\}^j) \) by duality, we see that the shuffle regularised \( \zeta(\{3,1\}^{n}, 1)  \) and \( \zeta(\{3,1\}^{n}, 1, 1) \) are as well.  This also holds true after switch from shuffle regularisation to stuffle regularisation (via the recipe established in  \cite{IKZ06}).   (Alternatively, one can differentiate the generating series expression \cite[Theorem 11.1]{BBBLspec} for \( \sum_n \Li_{\{1,3\}^n}(x) t^{4n} \), and apply know asymptotics of hypergeometric functions, c.f. \cite[Latter-half of proof of Theorem 3.5]{bachmannYamasakiCheckerboard}, and \cite[Remark, p.~13]{bowmanBradleyResolution}.)   Ultimately, one finds
		{ 
			\[
			\zeta^\star(1,\{1,3\}^{n+1}) = -\frac{1}{2 \cdot(-4)^n} \zeta(4n+5) + \frac{8}{3\zeta(2)} \sum_{\substack{a + b + 2c = 2n+2 \\ a, b, c \geq 0}} \!\!\!\!  \frac{(-1)^{a+c} (1-2^{2a})}{4^{a+b+c}} \zeta(2a)\zeta(2b)\zeta(4c+3) \,.
			\]}%
		Finally by applying the cyclic sum theorem \cite{ohnoWakabayashiCyclic,hoffmanAlgebraic} to the word \( (x^2y^2)^n \), one has (as was pointed out to SC by Michael Hoffman)
		\begin{equation}\label{eqn:cyclicsum}
		\zeta^\star(2,\{1,3\}^n,1,2) = -\zeta^\star(1,\{1,3\}^{n+1}) + 4 \zeta(4n+5) \,,
		\end{equation}
		so \( \zeta^\star(2,\{1,3\}^n,1,2) \) is also a polynomial in single-zeta values.  
		
		Using the stuffle-antipode of zeta-half values and the expressions from 
		{%
			\def\equationautorefname~#1\null{Equations~(#1)\null}\autoref{eqn:zh223aszs1313:a}%
		}%
		--\eqref{eqn:zh223aszs1313:d} (alternatively using a generalisation of techniques from \cite{hoffman2020yamamoto}), one can derive formulae for \( \zeta^\star(1,2,\{1,3\}^n) \) and \( \zeta^\star(\{2,3,\{1,3\}^n) \) in terms of \( \zeta^\star(2,\{1,3\}^n,1,2) \) and \( \zeta^\star(\{1,3\}^n,1,2) \), so the only essential evaluation remaining is that of \( \zeta^\star(\{1,3\}^n,1,2) \).
	\end{remark}
    
	\begin{remark}[Generalisations of the cyclic sum theorem?]
		From the explicit formula above for \( \zeta^\star(\{1,3\}^n, 1,2) \), and the following formula for \( \zeta^\star(3,\{1,3\}^n) \) derivable from \cite[Theorem 6]{munetaZS13} (c.f. also Theorem 9 \cite{hoffman2020yamamoto}),
		\[
			\zeta^\star(3,\{1,3\}^n) = - \frac{8}{3\zeta(2)} \sum_{\substack{a + b + 2c = 2n+1 \\ a,b,c \geq0}} \!\!\!\!  \frac{(-1)^{a+c} (1 - 2^{2a}) }{4^{a+b+c}} \zeta(2a)\zeta(2b)\zeta(4c+3) \,, 
		\]
		we see that the following combination simplifies dramatically
		\[
			\zeta^\star(\{1,3\}^n, 1,2) + \zeta^\star(3,\{1,3\}^n) =  \Big(4 - \frac{1}{16^n}\Big) \zeta(4n+3) = -4 \zeta(\overline{4n+3})\,.
		\]
		This is similar in form (but nevertheless different) to the cyclic sum identity arising in \autoref{eqn:cyclicsum}.  We venture to suggest there is some extension to the cyclic sum theorem in the setting of alternating interpolated MZV's, for which this is a special case.
	\end{remark}

	\begin{remark}[Proof via differential equations?]
		Broadhurst \cite[Theorem 11.1]{BBBLspec} showed that
		\[
			\sum_{n=0}^\infty \Li_{\{1,3\}^n}(x) t^{4n} = \pFq{2}{1}{\tfrac{1+\ii}{2} t, {-}\tfrac{1+\ii}{2} t}{1}{x}
			\pFq{2}{1}{\tfrac{1-\ii}{2} t, {-}\tfrac{1-\ii}{2} t}{1}{x} \,,
		\]
		as both sides are annihilated by the differential operator
		\[
			\Big( (1-x)\frac{\dd}{\dd x}\Big)^2 \Big( x \frac{\dd}{\dd x}\Big)^2 - t^4 \,,
		\]
		and have the same series expansion to order \( O(x^5) \).  By taking \( \frac{\dd}{\dd x} \), and applying the asymptotic formula (c.f. \cite[Eqn. (24)]{BBBLspec}) for \( x \to 1^- \),
		\[
			\pFq{2}{1}{1+z,1-z}{2}{x} = \frac{2 \gamma - \psi(1+z) - \psi(1-z) - \log(1-x)}{\Gamma(1+z)\Gamma(1-z)} + O((1-x)\log(1-x))  \,,
		\]
		and taking \( x \to 1^- \), one obtains a generating series expression for
		\[
			\sum_{n=0}^\infty \zeta(\{1,3\}^{n}, 1,2) t^{4n+4}
		\]
		in terms of trigonometric and digamma functions, showing \( \zeta(\{1,3\}^n, 1,2) \) is a polynomial in Riemann zeta values.  Can one give a similar differential equation proof for the evaluation of \( \zeta^\star(\{1,3\}^n) \) and more importantly \( \zeta^\star(\{1,3\}^n,1,2) \), by establishing a hypergeometric$^{(?)}$ evaluation for
		\[
			f(x,y) \coloneqq \sum_{n=0}^\infty \Li^\star_{\{1,3\}^n}(x) t^{4n}  \,,
		\]
		which satisfies
		\[
			\Bigl( (1-x)x \frac{\dd}{\dd x} \Bigr)^2  \Bigl( x \frac{\dd}{\dd x} \Bigr)^2 f(x,y) - t^4 f(x,y) = -t^4 (1-x) \,.
		\]
		(Here \( \Li^\star \) is the natural ``star-ification'' of the classical multiple polylogarithms, c.f. \( \operatorname{Le}_{s_!,\ldots,s_d}(t) \) in \S10.6, the proof of Lemma 10.6.1, \& Equation (10.14) \cite{ZhaoBook}.)
		This could allow a more conceptual proof for the evaluation of \( \zeta^\star(\{1,3\}^n,1,2) \) bypassing the non-conceptual creative telescoping approach above.
	\end{remark}

	Using the above evaluations, and properties of MZ$\half$V's, we can derive another apparently new MZSV evaluation involving the repeating \( \{1,3\}^n \) string.

\begin{corollary}\label{cor:zs1313p2}
	The following evaluation holds
	\begin{align*}
		&\zeta^\star(1,2,\{1,3\}^n,1,2) =  \\
		&\frac{16}{3\zeta(2)}  \sum_{\substack{a + b + 2d + 2j = 2n+2 \\ a, b, d, \geq 0 , j \geq 1}} \frac{(-1)^{b+d}(2 - (-4)^{-j})(1 - 4^a)}{2^{2a + 2b + 2d}} \zeta(2a)\zeta(2b)\zeta(4d+3) \cdot \zeta(4j+1) \\
		& - \frac{8}{3\zeta(2)} \sum_{\substack{a + b + 2d + 2j = 2n+1 \\ a, b, d, j \geq 0}} \frac{(-1)^{b+d}(4 - 4^{-2j})(1 - 4^a)}{2^{2a + 2b + 2d}} \zeta(2a)\zeta(2b)\zeta(4d+3) \cdot \zeta(4j + 3) \\
		&-2\sum_{\substack{j+k = n+1 \\ j,k \geq 1}} (2 - (-4)^{-j})(2 - (-4)^{-k}) \zeta(4j+1)\zeta(4k+1) \\
		&+ \frac{1}{2} \sum_{\substack{j+k = n+1 \\ j,k \geq 0}} \big\{ (4 - 4^{-2j})(4 - 4^{-2k}) + (-4)^{-j-k} \big\} \zeta(4j+3)\zeta(4k+3)  \,.
	\end{align*}
	
	\begin{proof}
		From the stuffle-product of zeta-half values (c.f. \cite[\S3]{HoffmanQuasi2020}), we have
		\[
			2 \delta_{\text{$n\,$even}} \zeta^\half(3,\{\overline2\}^{n}, 3) = \sum_{i=0}^n (-1)^i \zeta^\half(\{\overline2\}^i, 3) \zeta^\half(\{\overline2\}^{n-i}, 3) \,.
		\]
		By \autoref{cor:zh2223}, \( \zeta^\half(\{\overline2\}^{i}, 3) \) is a polynomial in Riemann zeta values, and therefore so is the whole right hand side.  Finally, by the generalised two-one theorem (c.f. \cite{ZhaoIdentity16}, \cite{CharltonCyclic20}), we have
		\[
			\zeta^\half(3,\{\overline2\}^{2n},3) = 2^{-2n-2} \zeta^\star(1,2,\{1,3\}^n,1,2)
		\]
		so the result follows.
	\end{proof}
\end{corollary}

\begin{remark}
	A similar result for \( \zeta^\half(3,\{\overline2\}^{2n+1}, 3) = -2^{-2n-3} \zeta^\star(2,\{3,1\}^{n+1},2) \), fails, as already \[
 \zeta^\half(3,\overline2,3) = -2^{-3} \zeta^\star(2,3,1,2) = -\frac{27}{40} \zeta (3,5)+\frac{1}{16} \pi ^2 \zeta (3)^2-\frac{9}{4} \zeta (5) \zeta (3)+\frac{1769 \pi ^8}{10368000} 
\] involves the (apparently) irreducible weight 8 depth 2 element \( \zeta(3,5) \).
\end{remark}

\begin{remark}[Categorising all \( \zeta^\star(\ldots, \{1,3\}^n, \ldots) \) evaluations]
	 Recall the well-known results \cite{munetaZS13} that \( \zeta^\star(\{1,3\}^n) \), \( \zeta^\star(3,\{1,3\}^n) \) are polynomials in single-zeta values (c.f. \autoref{remark:mzvtechniques}), and a corresponding result  \cite{pilehrood123,tasakaYamamoto,imatomiSome} for \( \zeta^\star(2,\{1,3\}^n) = -2^{2n+1} \zeta^\half(\{\overline2\}^{2n+1}) \) (which then also follows by \cite{HoffmanIhara17}; the conversion to MS$\half$V's is via the generalised two-one theorem \cite{ZhaoIdentity16}).  In \autoref{remark:mzvtechniques} we also presented a similar result for \( \zeta^\star(1,\{1,3\}^n) \).
	Combining these, with the results of \autoref{cor:zs1313} and \autoref{cor:zs1313p2} we obtain the following list seems, which seems to exhaust the \emph{convergent} MZSV's involving repeating \( 1,3 \) arguments which evaluate as single-zeta polynomials.
	\[
		\begin{cases}
		\begin{aligned}[t]
			\, & \zeta^\star(\{1,3\}^n) \,, \quad \zeta^\star(3,\{1,3\}^n) \,, \quad 
			\zeta^\star(2,\{1,3\}^n) \,, \quad  \zeta^\star(1,\{1,3\}^n) \,, \\[0.5ex]
			&\zeta^\star(\{1,3\}^n,1,2) \,,\quad 
			\zeta^\star(2,\{1,3\}^n,1,2) \,,\quad 
			\zeta^\star(1,2,\{1,3\}^n) \,,\quad 
			\zeta^\star(2,3,\{1,3\}^n) \,, \quad \\[0.5ex]
			&\zeta^\star(1,2,\{1,3\}^n,1,2)  \,
		\end{aligned}
		\end{cases}
	\]
	
	There are a number of further evaluations, which can be obtained for stuffle regularised MZSV's which have trailing 1's.  In particular, we have that
	\[
		\begin{cases}
		\begin{aligned}[t]
		& \, \reg_{\ast,T} \zeta^\star(\{1,3\}^n,1) \,, \quad \reg_{\ast,T} \zeta^\star(3,\{1,3\}^n,1) \,, \quad \\
		&
		\reg_{\ast,T} \zeta^\star(2,\{1,3\}^n,1) \,, \quad  \reg_{\ast,T} \zeta^\star(1,\{1,3\}^n,1) 
		\end{aligned}
		\end{cases}
	\]
	are polynomials in Riemann zeta values and \( T \), the regularisation parameter.  Again this list seems to be exhaustive, for stuffle regularised MZSV's with at least one trailing 1.  We briefly sketch how to show these results.\smallskip
	
	\paragraph{\bf Case \(\reg_{\ast,T}  \zeta^\star(\{1,3\}^{n},1) \):}  Apply the stuffle-antipode to \( \zeta^\star(\{1,3\}^n) \), to write
	\[
		\sum_{i=0}^n \reg_{\ast,T} \zeta(\{3,1\}^i) \cdot \zeta^\star(\{1,3\}^{n-i}) - \sum_{i=0}^{n-1} \zeta(3,\{1,3\}^i) \cdot \reg_{\ast,T} \zeta^\star(\{1,3\}^{n-1-i}, 1) = 0 \,.
	\]
	From the known evaluations of \( \reg_{\ast,T} \zeta(\{3,1\}^j) \,, \zeta^\star(\{1,3\}^j) \), and \( \reg_{\ast,T} \zeta(\{1,3\}^{\ell},1) \), the claim follows. \smallskip
	
	\paragraph{\bf Case \( \reg_{\ast,T}  \zeta^\star(3,\{1,3\}^{n-1},1) \)} Apply the stuffle-antipode to \( \reg_{\ast,T} \zeta^\star(\{3,1\}^n) \) to write
	\[
	\sum_{i=0}^n\zeta(\{1,3\}^i) \cdot \reg_{\ast,T} \zeta^\star(\{3,1\}^{n-i}) - \sum_{i=0}^{n-1} \reg_{\ast,T} \zeta(\{1,3\}^i,1) \cdot \zeta^\star(3,\{1,3\}^{n-1-i}) = 0 \,.
	\]
	From the known evaluations of \(  \zeta(\{1,3\}^j) , \reg_{\ast, T} \zeta^\star(\{3,1\}^k) \), and \( \reg_{\ast,T} \zeta(\{1,3\}^{\ell},1) \), the claim follows. \smallskip

	\paragraph{\bf Case \( \reg_{\ast, T} \zeta^\star(1, \{1,3\}^n, 1\):} Apply the stuffle-antipode to \( \reg_{\ast,T} \zeta^\star(1,\{1,3\}^n,1) \) to write
	\begin{align*}
		\sum_{i=0}^n \zeta(\{1,3\}^i) \cdot \reg_{\ast,T} \zeta^\star(1, \{1,3\}^{n-i},1)  - \sum_{i=1}^n \reg_{\ast,T} \zeta(\{1,3\}^i, 1) \cdot  \zeta^\star(1,\{1,3\}^{n-i}) & \\[-1ex]
		+ \zeta(\{1,3\}^n, 1, 1) & = 0 \,.
	\end{align*}
	From the known evaluations of \( \zeta(\{1,3\}^j) , \reg_{\ast,T} \zeta(\{1,3\}^k,1) , \zeta^\star(1,\{1,3\}^\ell)  \) and \( \reg_{\ast,T} \zeta(\{1,3\}^m, 1,1) \) (see \autoref{eqn:z1311} above), the claim follows.\smallskip
	
	\paragraph{\bf Case \( \reg_{\ast,T} \zeta^\star(2,\{1,3\}^n, 1) \):} Apply the stuffle-antipode to \( \zeta^\star(2,\{1,3\}^n, 1) \) to write
		\begin{align*}
		\sum_{i=0}^n \zeta(\{1,3\}^i) \cdot \reg_{\ast,T} \zeta^\star(2, \{1,3\}^{n-i},1)  - \sum_{i=1}^n \reg_{\ast,T} \zeta(\{1,3\}^i, 1) \cdot  \zeta^\star(2,\{1,3\}^{n-i}) & \\[-1ex]
		+ \zeta(\{1,3\}^n, 1, 2) & = 0 \,.
		\end{align*}
		From the known evaluations of \( \zeta(\{1,3\}^j) , \reg_{\ast,T} \zeta(\{1,3\}^k,1) , \zeta^\star(2,\{1,3\}^\ell)  \) and \(  \zeta(\{1,3\}^m, 1,2) \), the dual of \( \zeta(3,\{1,3\}^m) \), the claim follows.
\end{remark}

We now turn to the proof of \autoref{theorem:6F5identity}, which implies all of the above MZSV and MZ$\half$V evaluations.

\begin{proof}[Proof of \autoref{theorem:6F5identity}]
	
We give a proof of this Theorem via Wilf-Zeilberger pairs, which is a special case of creative telescoping. 

Introduce
\begin{align*}
	F(x,y,m) & \coloneqq F_\text{even}(x,y,m) + F_\text{odd}(x,y,m) \\
	&= \begin{aligned}[t]
	& \frac{y^2(4m+3)(4x^2 - 1)(y^2-1)}{8(m+1)(2m+1)(x^2 - 1)(4y^2-1)\cdot ((2m+1)^2-4x^2) \cdot ((m+1)^2-y^2)}
	\\ 
	& \hspace{2em} \cdot (4m^2 + 6m - 4x^2 + 3) \cdot  \frac{\poch{\tfrac{3}{2}-x}{m}\poch{\tfrac{3}{2}+x}{m}\poch{2-y}{m}\poch{2+y}{m}}{\poch{2-x}{m}\poch{2+x}{m}\poch{\tfrac{3}{2}-y}{m}\poch{\tfrac{3}{2}+y}{m}} \,.
	\end{aligned}
	\end{align*}
Let $g(x,y)$ denote the digamma and trigonometric expression in statement of \autoref{theorem:6F5identity}. We need to prove that $f(x,y) := \sum_{m\geq 0} F(x,y,m) = g(x,y)$. To this end, let $C(x,y) = f(x,y) - g(x,y)$. \par

\begin{lemma}
The functions $f(x,y)$ and $g(x,y)$ are analytic with possible isolated poles only at $(x,y)\in \mathbb{Z} \times(\mathbb{Z}+\tfrac{1}{2})$.
\end{lemma}
\begin{proof}
For $g(x,y)$ this is easy to check. For $f$, its poles come from the terms $\poch{2\pm x}{m}, \poch{\tfrac{3}{2}\pm y}{m}$ in denominator of $F(x,y,m)$ (or from \( x^2 - 1 \), or \( 4y^2 - 1 \)) resulting in a pole whenever $ 2 \pm x, \tfrac{3}{2} \pm y $ are negative integers.  
\end{proof}

We remark that the series $\sum_{m\geq 0} F(x,y,m)$ is absolutely convergent for all $x,y\in \mathbb{C}$, (ignoring issue near those poles), as $F(x,y,m)$ is $O(m^{-3})$ for large $m$. 

\begin{lemma}
The function $C(x,y) = f(x,y) - g(x,y)$ is analytic at $(x,y)=(0,\tfrac{1}{2})$. 
\end{lemma}
\begin{proof}
We calculate the Laurent expansion of $f$ and $g$ at $(x,y)=(0,\tfrac{1}{2})$.  For $g$ this is easy,
\[
g(x,y) = \frac{\frac{1}{2 \pi ^2}-\frac{1}{8}}{y-\frac{1}{2}} + O(1) \,.
\]
For $f$, note that we already have $\tfrac{1}{2y-1}$ (from the denominator factor $4y^2-1$) in the expression of $f$, so around \( (x,y) = (0, \tfrac{1}{2}) \) we have
\[
f(x,y) = \frac{1}{2y-1} \sum_{m\geq 0} \underbrace{\frac{- 3 (4 m+3) \left(4 m^2+6 m+3\right) }{256 (m+\tfrac{1}{2}) (m+1) (m+\tfrac{3}{2}) (2 m+1)^3 } \cdot \frac{\poch{\tfrac{3}{2}}{m}^3 \poch{\tfrac{5}{2}}{m}}{\poch{1}{m} \poch{2}{m}^3}}_{\eqqcolon a_m} + O(1) \,.
\]
one checks
\[
a_m = \Delta_m \biggl( - \frac{3 (m+1)^3 (4 m^2+2 m+1) }{4 (2 m+1)^4 (2 m+3)} \cdot \frac{\poch{\tfrac{3}{2}}{m}^3 \poch{\tfrac{5}{2}}{m}}{ \poch{1}{m} \poch{2}{m}^3 } \biggr) := \Delta_m b_m \,,
\]
where \( \Delta_m \coloneqq \mathbf{S}_m - 1 \) denotes the forward difference operator, i.e. \( \Delta_m h(m) = h(m+1) - h(m) \).  So the sum telescopes and one finds
\[
	\sum_{m\geq 0} a_m = \lim_{m\to\infty} b_{m} - b_0 = -\frac{1}{4} - \frac{1}{\pi^2} \,.
\]
Thus around \( (x,y) = (0, \tfrac{1}{2}) \), we also have \[
	f(x,y) = \frac{\frac{1}{2 \pi ^2}-\frac{1}{8}}{y-\frac{1}{2}} + O(1)
\]
whence \( C(x,y) = f(x,y) - g(x,y) \) is indeed analytic there as claimed.
\end{proof}

By \autoref{modgrowth_lemma}, both $g(x,y),f(x,y)$ have moderate growth along the imaginary direction if we fix one variable, and let the other vary.  The following observation is the core of our argument.
\begin{lemma}
	The function \( C(x,y) \) is $1$-periodic in each variable, i.e.
\[
C(x+1,y) = C(x,y) \,, \qquad C(x,y+1) = C(x,y) \,.
\]
\end{lemma}
\begin{proof}
Let $\Delta_x = \mathbf{S}_x - 1$ be the forward difference operator, i.e. $\Delta_x h(x) = h(x+1)-h(x)$. We need to prove $\Delta_x f(x,y) = \Delta_x g(x,y)$ and $\Delta_y f(x,y) = \Delta_y g(x,y)$. Write \begin{align*}
r_1(x,y,n) = & {} -\frac{(m+1) (2 m+1) (m-x+1) ((2m+1)^2 - 4y^2)}{(4 m+3) x (2 m-2 x-1) (4 m^2+6 m-4 x^2+3) } \\
& \cdot
\biggl( \frac{4m^2 + 2m - 4x^2 - 8x - 9}{(x+1)^2 - y^2} + \frac{2m^2 + 4m^2 x + 2m x + m - 5x - 8y^2 - 4xy^2 - 1}{(x^2 - y^2)((x+1)^2 - y^2)} \biggr) \\[1ex]
r_2(x,y,n) = {}
& \frac{(m+1) (2 m+1) (m-x+1) (m+x+1) (2 m-2 y+1)}{(4 m+3) y (4 m^2+6 m-4 x^2+3 ) } \\
& \cdot \biggl( \frac{4m^2 + 2m - 4x^2 + 4y + 3}{(x^2 - (y+1)^2)(m-y)} - \frac{4y^2 + 4m y + 4y + 2m + 1}{(x^2 - y^2)(x^2 - (y+1)^2)} \biggr) \,.
\end{align*}
One checks \[
\Delta_x F(x,y,m) = \Delta_m (r_1(x,y,m) F(x,y,m)) \,, \qquad \Delta_y F(x,y,m) = \Delta_m (r_2(x,y,m) F(x,y,m))
\]
so 
\begin{align*}
\Delta_x f(x,y) = \sum_{m\geq 0} \Delta_x F(x,y,m) &= \sum_{m\geq 0} \Delta_m (r_1(x,y,m) F(x,y,m)) \\ &= \lim_{m\to \infty} r_1(x,y,m)F(x,y,m) - r_1(x,y,0)F(x,y,0) \,.
\end{align*}
One calculates the limit (it gives factor $\cot \pi x \tan \pi y$ that appears in $g(x,y)$).  Then by comparing the above to $\Delta_x g(x,y)$, which is readily calculable,  one sees -- after some simplification of digamma functions -- that is agrees with \( \Delta_x f(x,y) \).  So \( \Delta_x g(x,y) = \Delta_x f(x,y) \) or equivalently \(  C(x+1,y) = C(x,y) \). For the variable $y$ the calculation is analogous.  This completes the proof of the lemma.
\end{proof}

Finally we argue $C(x,y) = 0$ identically.  The 1-periodicity and holomorphicity at $(0,\tfrac{1}{2})$ implies $C(x,y)$ is holomorphic at each $(x,y) \in  \mathbb{Z} \times (\mathbb{Z}+\tfrac{1}{2})$, so $C$ is an entire function. Fix a generic $y$, consider $C(x,y)$ as an entire function of $x$; $C(x,y)$ being of moderate growth implies it is a constant in $x$. Similarly, it is constant in $y$.  Thus $C(x,y)$ is a constant independent of $x$ or $y$. 

Since $f(x,0) = 0$ (note the factor $y^2$ in front of $F(x,y,m)$) and $g(x,0) = 0$, this constant must be $0$. This completes the proof of \autoref{theorem:6F5identity}.
\end{proof}

\section{Results on Hoffman's multiple \texorpdfstring{$t$}{t}-value \texorpdfstring{$t^{\half}(2,\{1\}^n,2)$}{t\textasciicircum{}½(2, \{1\}\textasciicircum{}n, 2)}}
\label{sec:th21112}

In this last section, we will prove the following evaluation on interpolated multiple $t$-values.
\begin{theorem}\label{theorem:tvalue}
	The following evaluation holds for any \( n \geq 0 \in \mathbb{Z} \),
\[
t^{\half}(2,\{1\}^n,2) = \frac{n+3}{2^{n+2}} t(n+4)  \, =  \frac{n+3}{2^{n+2}} \cdot \Bigl(1 - \frac{1}{2^{n+4}}\Bigr) \zeta(n+4) \,  \,. 
\]
\end{theorem}

\begin{remark}
	This resolves Conjecture 3.10 from \cite{charltonHoffmanMtVSym}, which posed the evaluation in the case of \( n \) odd.  The case \( n \) even was already handled therein, using a symmetry theorem for MtV's to establish a general generating series expression for \( t^\half(2\ell+2, \{1\}^{2n}, 2\ell+2) \).  In the case of \( n \) odd, \( t(n+4) \) is expected to be irreducible, so the symmetry theorem seemed to be insufficient to resolve this case.  It was also noted in \cite{charltonHoffmanMtVSym} that evaluations of \( t^\half(2\ell+2, \{1\}^{2n+1}, 2\ell+2) \), \( \ell > 0 \), apparently involved higher depth irreducibles, so no simple formulae are expected beyond this case.
\end{remark}

\begin{proof}[Proof of \autoref{theorem:tvalue}]
Readers interested in computational details for this example should consult the \textsf{Mathematica} notebook attached. To begin with, we first write down the generating function of our MtV's of interest (or rather a small modification thereof), using the strategy of \autoref{sec:genser}.  We have
\begin{align*}
&\sum_{k\geq 0} 2^{k+4} \bigl( t^r(2,\{1\}^k,2) - r^{k+1}  t(4+k) \bigr) (2x)^k \\
& = \sum_{\substack{n = 0 \\ m = 0}}^\infty \begin{aligned}[t]  \frac{1}{(m + \tfrac{1}{2})^2} & (1+ur+u^2r^2+\cdots) \\[-1ex]
& \cdot \prod_{m<\ell<n}\left(1+s+s^2r+s^3r^2+\cdots \right) 
\cdot  \frac{1}{(n + \tfrac{1}{2})^2}(1+vr+v^2r^2+\cdots)  \end{aligned} \\
& = \sum_{\substack{n = 0 \\ m = 0}}^\infty \frac{1}{(m+\tfrac{1}{2})^2 (1 - r u)} \cdot \prod_{m < \ell < n} \Bigl( 1 - \frac{s}{1 - s r} \Bigr) \cdot \frac{1}{(n+\tfrac{1}{2})^2 (1 - r v)} \,,
\end{align*}
where inside the summation
\[
u=\frac{2x}{m + \tfrac{1}{2}}, \quad v=\frac{2x}{n  + \tfrac{1}{2}}, \quad s=\frac{2x}{\ell + \tfrac{1}{2}} \,.
\]
When the interpolation parameter is $r=\half$, the right hand side further simplifies into the following double hypergeometric series
\[
\sum_{n>0} \sum_{m=0}^{n-1} \frac{1}{(n+\tfrac{1}{2})(m+\tfrac{1}{2})} \frac{\Gamma(\tfrac{1}{2}+m-x)\Gamma(\tfrac{1}{2}+n+x)}{\Gamma(\tfrac{3}{2}+n-x)\Gamma(\tfrac{3}{2}+m+x)} \eqqcolon F(x)
\]

First we note that the above series (in $n$) converges absolutely\footnote{if we view $F(x)$ as a double series, then absolute convergence is only for $|\Re(x)|<1/2$.} for $\Re(x) < 1/2$. Indeed, we have \[
\frac{\Gamma(\tfrac{1}{2}+m-x)\Gamma(\tfrac{1}{2}+n+x)}{\Gamma(\tfrac{3}{2}+n-x)\Gamma(\tfrac{3}{2}+m+x)}
 = O(m^{-1-2x}) O(n^{-1+2x}) \,,
 \]
so the inner sum is $O(n^{-1+2x}) \sum_{m\geq 0}^{n-1} \frac{O(m^{-1-2x})}{m+ 1/2} = O(n^{-1+2x}) O(n^{\max(0,-1-2x)})$.  Finally, if it holds that $2x-1 + \max(0,-1-2x) < 0 $, (which is the case for $x<\tfrac{1}{2}$), we have
\[
\sum_{n>0} \frac{1}{n+\tfrac{1}{2}} O(n^{-1+2x}) O(n^{\max(0,-1-2x)}) < \infty \,.
\]
 
 We also need asymptotic behaviour of the inner sum.
\begin{lemma}\label{lem:gasym}
Let 
\[
g(n,x) \coloneqq \frac{1}{n+\tfrac{1}{2}} \frac{\Gamma(\tfrac{1}{2}+n+x)}{\Gamma(\tfrac{3}{2}+n-x)} \sum_{m = 0}^{n-1} \frac{\Gamma(m+\tfrac{1}{2}-x)}{(m+\tfrac{1}{2})\Gamma(\tfrac{3}{2}+m+x)}
\]
then for $\Re(x)$ sufficiently negative, 
\[
g(n,x) = \frac{-1}{(1+2x)n^3} + \frac{2+x}{(1+2x)n^4} + O(n^{-5}) \,.
\]
\end{lemma}
\begin{proof}
This follow easily by Euler-Maclaurin summation. For $\Re(x)>1$, we have $$\sum_{m=0}^{n-1} m^x = \frac{n^{x+1}}{x+1} + \frac{n^x}{2} + O(n^{x-1})$$
From
$$\frac{\Gamma(m+1/2-x)}{(m+1/2)\Gamma(3/2+m+x)} = m^{-2x-2} + (-1-x)m^{-2x-3} + O(m^{-2x-4})$$
it follows that, if $\Re(x)$ is sufficiently negative (i.e. $\Re(-2x-3) > 0$), we have
\begin{align*}
    &\sum_{m= 0}^{n-1} \frac{\Gamma(m+1/2-x)}{(m+1/2)\Gamma(3/2+m+x)} \\
    &= (-x-1) \left(\frac{n}{-2 x-2}+\frac{1}{2}\right) n^{-2 x-3}+\left(\frac{n}{-2 x-1}+\frac{1}{2}\right) n^{-2 x-2} + O(n^{-2x-3})\end{align*}
multiplying with the asymptotic expansion of $\frac{\Gamma(1/2+n+x)}{(n+1/2)\Gamma(3/2+n-x)}$ gives the result. 
\end{proof}

Finally we come to the algebraic part of the proof.  Creative telescoping (calling \texttt{Annihilator} to find the annihilator ideal of the inner sum, and then \texttt{CreativeTelescoping} to compute the creative telescoping relations for this annihilator ideal) says the operator
\[
(x+2)^3 \mathbf{S}_x^2-2 (x+1)^3 \mathbf{S}_x+x^3 + (\mathbf{S}_n-1)P
\] kills $g(n,x)$, where 
{\small \begin{align*}
	P = {} & \begin{aligned}[t] 
	\frac{1+2n}{n} \biggl( 
	\frac{2n^3 - 6n^2 + 3n -2}{8} & {} - \frac{(7 + 14n+4n^2)x}{16} - \frac{nx^2}{4} - \frac{x^3}{4} \\
	& + \frac{2n^3 + 2n^2 - n - 1}{8(3+2x)} - \frac{4n^3 - 8n^2 + 5n - 1}{4(3-2n+2x)} \biggr) \mathbf{S}_x \end{aligned}  \\
	& - \begin{aligned}[t] 
	\frac{(1+2n)(1+2n-2x)}{n} \biggl( 
	&  \frac{8n^2 + 16n-1}{32} + \frac{(n-1)(2n-1)}{2(3 - 2n+2x)} + \frac{3x}{16} - \frac{x^2}{8} - \frac{4n^2-1}{64(x+1)^2} \\
	& - \frac{12n^2+1}{64(x+1)} - \frac{n^2}{8(2x+1)} + \frac{n^2-1}{4(3+2x)} 
	\biggr)  \,.
	\end{aligned}
	\end{align*}
}
Therefore applying this to \( g(n,x) \), and summing over \( n \) leads to 
\[
x^3 F(x)-2(x+1)^3 F(x+1)+(x+2)^3 F(x+2) + \lim_{n\to\infty} (Pg)(n,x) - (Pg)(1,x) = 0 \,,
\]
via the telescoping in \( n \).  Using the asymptotic behaviour of $g(n,x)$ given in \autoref{lem:gasym}, one can evaluate the limit (under the assumption that $\Re(x)$ is sufficiently negative), obtaining
\[
	\lim_{n\to\infty} (Pg)(n,x) = \frac{4(1+4x+2x^2)}{(1+2x)^2(3+2x)^2} \,.
\]
Overall, we therefore have that \( F \) satisfies the recurrence
\[
x^3 F(x)-2(x+1)^3 F(x+1)+(x+2)^3 F(x+2) + \frac{4 \left(2 x^2+4 x+1\right)}{(2 x+1)^2 (2 x+3)^2} = 0 \,,
\]
this holds when $\Re(x)$ is sufficiently negative. As $F(x)$ is analytic for $\Re(x)<1/2$, the above implies $F$ can be meromorphically continued to all of $\mathbb{C}$. This second order recurrence has the general solution
\[
\frac{2 \psi\bigl(x+\frac{1}{2}\bigr)-x \, \psi'\bigl(x+\frac{1}{2}\bigr)}{4 x^3}+\frac{c_1(x)}{x^3}+\frac{c_2(x)}{x^2} \,,
\]
where $c_i(x)$ are 1-periodic functions.  The recursion also implies that $F(x)$ has a double pole at $x=\tfrac{1}{2}$.  Since $F$ has moderate growth by \autoref{modgrowth_lemma}, we see both $c_i$ are of moderate growth and both have at most double poles at $\tfrac{1}{2} + \mathbb{Z}$. Therefore by \autoref{cot_poly}, there exists constants $u_i,v_i,w_i \in \mathbb{C}$ such that
\[
c_i(x) = u_i + v_i \tan(\pi x) + w_i \tan^2(\pi x) \,.
\]
We know, from first few values of \( t^\half(2,\{1\}^{k},2) - \frac{1}{2^{k+1}} t(k+4) \), obtained using the alternating MZV Data Mine \cite{mzvDM}, say, that
\[
F(x) = \frac{\pi^4}{24} + \frac{31}{2}\zeta(5) x+ \frac{\pi^6}{20}x^2 + \cdots \,.
\]
These three terms (and the analyticity of \( F(x) \) at \( x=0 \)) are already enough to allow us to fix these six constants.  We find, namely
\[
	\begin{cases} 
	\begin{alignedat}{6}
	\,\, & u_1 = 0\,, && v_1 = -\frac{\pi}{2}\,, \quad  && w_1 = 0 \,, \\
	\,\, & u_2 = 2 + \frac{\pi^2}{4}\,,\quad && v_2 = 0\,, && w_2 = \frac{\pi^2}{4} \,.
	\end{alignedat}
	\end{cases}
\]
So overall we have proven that
\[
F(x) = \frac{\gamma + 2 \log(2) + \psi(\tfrac{1}{2}+x)}{2x^3} - \frac{\psi'(\tfrac{1}{2}+x)}{4x^2} + \frac{3\pi^2}{8x^2} - \frac{\pi \tan{\pi x}}{2x^3} + \frac{\pi^2 \tan^2{\pi x}}{4x^2} \,.
\]
Since \( t(n) = (1 - \tfrac{1}{2^n}) \zeta(n) \), this is readily seen to be equivalent to the statement of \autoref{theorem:tvalue}.
\end{proof}

\begin{remark}
	The authors (the second in particular) would be interested to see a more ``conceptual'' proof of this evaluation.  The approach taken by SC in a recent note \cite{STpm2111pm1} on evaluations of multiple S and T values\footnote{Warning: the note  \cite{STpm2111pm1} employs the opposite MZV convention, wherein \( \zeta(2,1) \) converges.} of the form \(  S(\stackon[.1pt]{$2$}{\brabar}, 1, \ldots, 1, \stackon[.1pt]{$1$}{\brabar}) \) and \( T(\stackon[.1pt]{$2$}{\brabar}, 1, \ldots, 1, \stackon[.1pt]{$1$}{\brabar}) \), resolving Conjectures of Xu, Yan, and Zhao presented in \cite[Questions 1, and 2]{xu2022alternating}, suggests looking at the following iterated integral representation for the interpolated MtV,
	\begin{align*}
		& 2^{n+2} \, t^\half(2,\{1\}^n,2) \\
		 &= \int_0^1 \Bigl(-\frac{1}{t-1}  - \frac{1}{t+1}\Bigr)\mathrm{d}t \circ \frac{\mathrm{d}t}{t} \circ \Bigl(\Bigl( \frac{1}{t-1} + \frac{1}{t+1} - \frac{1}{t} \Bigr) \mathrm{d}t \Bigr)^{n+1} \circ \frac{\mathrm{d}t}{t} \\
		&= \int_{0 < t_1 < t_2 < t_3 < 1} \frac{-2}{(t_1-1)(t_1+1)} \frac{1}{t_2} \cdot \frac{\big( \log(t_2^{-1} - t_2) - \log(t_3^{-1} - t_3) \big)^{n+1}}{(n+1)!} \cdot \frac{1}{t_3} \, \dd t_1 \dd t_2 \dd t_3 \,,
	\end{align*}
	and trying to explicitly write down a primitive via \emph{fixed} depth multiple polylogarithms.  Unfortunately the current form of the integral does not seem to permit such a direct fixed depth evaluation.
\end{remark}

	\bibliographystyle{habbrv2} 
	\bibliography{ref.bib} 

\end{document}